\documentclass[12pt]{amsart}

\usepackage{amsfonts, amsthm}
\usepackage{enumitem}
\usepackage{indentfirst}
\usepackage{longtable}
\usepackage[all, cmtip]{xy}
\usepackage{graphicx}
\usepackage{epstopdf}
\usepackage{array}

\usepackage{hyperref}

\hypersetup{colorlinks,linkcolor=blue,citecolor=red,urlcolor=blue,pagebackref,hyperindex}

\usepackage{tikz}
\usepgflibrary{shapes.geometric}

\numberwithin{equation}{section}

\def\xh{\\[5pt]}

\def\Q{\mathbb{Q}}
\def\A{\mathcal{A}}
\def\D{\mathcal{D}}
\def\P{\mathcal{P}}
\def\Qcal{\mathcal{Q}}
\def\M{\mathcal{M}}
\def\L{\mathcal{L}}

\def\G{\mathcal{G}}

\def\vf{\mathfrak{v}}
\def\wf{\mathfrak{w}}
\def\bfa{{\bf a}}
\def\bfd{{\bf d}}
\def\lqv{n}

\def\QQ{\mathbf{Q}}

\def\dis{\displaystyle}

\renewcommand{\th}[1]{$#1^{\text{th}}$}

\makeatletter
\renewenvironment{proof}[1][\proofname]{\par
	\pushQED{\qed}%
	\normalfont \partopsep=\z@skip \topsep=\z@skip
	\trivlist
	\item[\hskip\labelsep
	\itshape
	#1\@addpunct{.}]\ignorespaces
}{%
\popQED\endtrivlist\@endpefalse
}
\makeatother

\tikzstyle{MyPic}=[draw=lightgray, inner sep=2pt]
\tikzstyle{SoThick}=[solid, very thick, line width=2pt, draw=black]
\tikzstyle{ThickDashed}=[densely dashed, line width=1pt, draw=black]
\tikzstyle{MyNode}=[circle, draw=black, fill=black, thick, minimum size=2mm]

\def\tsomepoints{
\coordinate (P00) at (0,0);
\coordinate (P01) at (0,1);
\coordinate (P02) at (0,2);
\coordinate (P03) at (0,3);
\coordinate (P04) at (0,4);
\coordinate (P05) at (0,5);

\coordinate (P10) at (1,0);
\coordinate (P11) at (1,1);
\coordinate (P12) at (1,2);
\coordinate (P13) at (1,3);
\coordinate (P14) at (1,4);
\coordinate (P15) at (1,5);

\coordinate (P20) at (2, 0);
\coordinate (P21) at (2, 1);
\coordinate (P22) at (2, 2);
\coordinate (P23) at (2, 3);
\coordinate (P24) at (2, 4);
\coordinate (P25) at (2, 5);

\coordinate (P30) at (3, 0);
\coordinate (P31) at (3, 1);
\coordinate (P32) at (3, 2);
\coordinate (P33) at (3, 3);
\coordinate (P34) at (3, 4);
\coordinate (P35) at (3, 5);

\coordinate (P40) at (4, 0);
\coordinate (P41) at (4, 1);
\coordinate (P42) at (4, 2);
\coordinate (P43) at (4, 3);
\coordinate (P44) at (4, 4);
\coordinate (P45) at (4, 5);

\coordinate (P50) at (5, 0);
\coordinate (P51) at (5, 1);
\coordinate (P52) at (5, 2);
\coordinate (P53) at (5, 3);
\coordinate (P54) at (5, 4);
\coordinate (P55) at (5, 5);

\coordinate (P60) at (6, 0);
\coordinate (P61) at (6, 1);
\coordinate (P62) at (6, 2);
\coordinate (P63) at (6, 3);
\coordinate (P64) at (6, 4);
\coordinate (P65) at (6, 5);

\coordinate (P70) at (7, 0);
\coordinate (P71) at (7, 1);
\coordinate (P72) at (7, 2);
\coordinate (P73) at (7, 3);
\coordinate (P74) at (7, 4);
\coordinate (P75) at (7, 5);

\coordinate (P80) at (8, 0);
\coordinate (P81) at (8, 1);
\coordinate (P82) at (8, 2);
\coordinate (P83) at (8, 3);
\coordinate (P84) at (8, 4);
\coordinate (P85) at (8, 5);

\coordinate (P90) at (9, 0);
\coordinate (P91) at (9, 1);
\coordinate (P1000) at (10, 0);
\coordinate (P1001) at (10, 1);
\coordinate (P1100) at (11, 0);
\coordinate (P1101) at (11, 1);
}

\theoremstyle{plain}

\newtheorem{theorem}{Theorem}[section]
\newtheorem{proposition}[theorem]{Proposition}
\newtheorem{lemma}[theorem]{Lemma}

\theoremstyle{definition}

\newtheorem{definition}[theorem]{Definition}
\newtheorem{example}[theorem]{Example}
\newtheorem{remark}[theorem]{Remark}

\newtheoremstyle{note}{0pt}{0pt}{}{}{\bfseries}{.}{5pt}{}
\theoremstyle{note}

\newtheorem{note}{Note}

\begin{document}
	
\title{New combinatorial formulas for cluster monomials of type $A$ quivers}

\author{Kyungyong Lee}
\address{Department of Mathematics \\
 University of Nebraska--Lincoln \\
 Lincoln, NE 68588, USA \\
and Korea Institute for Advanced Study \\
Seoul 02455, Republic of Korea}
\email{klee24@unl.edu; klee1@kias.re.kr}
\thanks{The first author was supported by the Korea Institute for
Advanced Study (KIAS), the AMS Centennial Fellowship, NSA grant
H98230-14-1-0323, and the University of Nebraska--Lincoln.}

\author{Li Li}
\address{Department of Mathematics and Statistics \\
 Oakland University \\
 Rochester, MI 48309}
\email{li2345@oakland.edu}
\thanks{The second author was partially supported by the
Oakland University URC Faculty Research Fellowship Award, and NSA grant H98230-16-1-0303.}

\author{Ba Nguyen}
\address{Department of Mathematics, Wayne State University, Detroit, MI 48202}
\email{ba.nguyen@wayne.edu}

\begin{abstract}
Lots of research focuses on the combinatorics behind various bases of cluster algebras.
This paper studies the natural basis of a type $A$ cluster algebra, which consists of all cluster monomials. We introduce a new kind of combinatorial formulas for the cluster monomials in terms of the so-called globally compatible collections. We give bijective proofs of these formulas by comparing with the well-known combinatorial models of the $T$-paths and of the perfect matchings in a snake diagram. For cluster variables of a type $A$ cluster algebra, we give a bijection that relates our new formula with the theta functions constructed by Gross, Hacking, Keel and Kontsevich.
\end{abstract}

\maketitle


\section{Introduction}

Cluster algebras were first introduced by S. Fomin and A. Zelevinsky in \cite{FZ1} to design an algebraic framework for understanding total positivity and canonical bases for quantum groups. 
A cluster algebra is a subring of a rational function field generated by a distinguished set of Laurent polynomials called cluster variables. The long-standing Positivity Conjecture, now proved in \cite{LS} and \cite{GHKK}, asserts that  the coefficients in the Laurent expansion of any cluster variable with
respect to any fixed cluster are positive integers.
From the combinatorial point of view, the Positivity Conjecture suggests that these coefficients should count some combinatorial objects.
Lots of research focuses on building such combinatorial models. We give a brief summary of the pros and cons of four such models.
\begin{itemize}[leftmargin=20pt] \itemsep=5pt
\item \textbf{$T$-paths:} In \cite{S1}, Schiffler (independently Carroll--Price and Fomin--Zelevinsky in
their unpublished work) obtained a formula for the cluster variables of a cluster algebra of finite type $A$ (see \S2 for the definition)
in terms of $T$-paths. This formula has been modified and generalized to cluster algebras coming from surfaces \cite{MS, MSW, ST, S2, GM}. The formula is computation-friendly, but it does not seem to generalize to cluster algebras not coming from surfaces.

\item
\textbf{Perfect matchings of a snake diagram:} A description that is similar to the $T$-path model but has a more graph-theoretic flavor \cite{MS}. Interesting combinatorics, for example the snake graph calculus \cite{CS1,CS2}, arises in the study of this model. This formula is simply bijective to the $T$-path formula, but uses more classical graph-theoretical notion ``perfect matching'' and is easier to compute; like the $T$-path formula, it is also restricted to cluster algebras coming from surfaces.

\item
\textbf{Compatible pairs in a Dyck path:} In \cite{ls-comm}, the cluster variables of rank 2 quivers are described in terms of Dyck paths. A more general construction of the so-called compatible pairs is used in the study of greedy bases in \cite{LLZ}, and another generalization called GCC is used in \cite{B9, LLM}. This formula is computation-friendly and easy to implement,  and it applies to rank 2 cluster algebras that do not necessarily come from surfaces. Meanwhile, the main drawback is that we could not yet find a generalization that gives the cluster variables for higher rank non-type-$A$  cluster algebras. The combinatorics is quite different than the $T$-paths and perfect matchings, but we shall give a bijection in this paper showing that they indeed coincide in the type $A$ case (Theorem \ref{PMsGCCs} and Theorem \ref{TpathsGCCs}).

\item
\textbf{Broken lines and Theta functions:} Discovered in \cite{GHKK}, they are the most general combinatorial models so far. They are so powerful that can be used to proved several well-known conjectures including the positivity conjecture. On the other hand, they are mainly of theoretical importance but do not yet give a satisfying combinatorial model (at least not in the sense of the previous three models): for example, the finiteness of the number of broken lines is not immediate from the definition, and  it is difficult to implement even for rank 2 cluster algebras.
\end{itemize}

Our ultimate goal is to find a combinatorial model that is both general and effective in computation. Even though this goal appears out of reach for now, we feel that the model of maximal Dyck paths and compatible pairs has the potential to be generalized. This motivates the main goal of this paper:

{\it For a type $A$ quiver, give a new formula for the cluster monomials using a combinatorial model similar to compatible pairs, and find the bijections to other known models.}

We reach this goal by proving three equivalent formulas.

-- In Theorem \ref{thm:01sequence}, we give a formula using a sequence of 0-1 sequences called a GCS (globally compatible sequence), where each vertex of the quiver is assigned a 0-1 sequence satisfying a certain compatibility condition.

-- In Theorem \ref{thm:GCCxa}, we give a formula using globally compatible collections (GCCs) in Dyck paths. This formula has a similar flavor to the combinatorial formula for greedy bases in \cite{LLZ}.

-- In \S\ref{dvecpipe}, we first use a combinatorial gadget called pipelines to decompose the \textbf{d}-vector of a cluster monomial into the ones of cluster variables, then give a formula for cluster variables using GCCs in Theorem \ref{MainThm}.

We would like to point that that the above results are extending the results on the equioriented type $A$ quivers given in \cite{B9}.

Moreover, for cluster variables of a type $A$ quiver, we construct a bijection between GCSs (which is equivalent to GCCs) and broken lines in Theorem \ref{MainThm2} (the even rank case) and \ref{MainThm2'} (the general case), which relates our new formula with the theta functions constructed in \cite{GHKK}. The simplicity of this bijection came as a surprise for us: namely, under our setting, the $i$-th number in a GCS (which is a 0-1 sequence) is 0 if and only if the corresponding broken line bends at the $i$-th coordinate hyperplane ${\bf e}_i^\perp$.  This suggests that there could be further connections between our new combinatorial formulas and theta functions, and thus could provide a new approach to understanding broken lines (which are difficult to describe explicitly in general); however, as explained in Remark \ref{conditions not needed}, right now we are only able to construct a bijection for cluster variables of a type $A$ quiver because the broken lines are special in this case.

The paper is organized as follows. In \S 2 we recall the definition of cluster algebra and some facts about type $A$ quivers. In \S 3 we define the  \textbf{d}-vector of a cluster monomial and introduce its decomposition using pipelines. \S4    consists of the statements of the main results of the paper. In \S5 we prove the GCC formula for cluster variables (Theorem \ref{MainThm}) by establishing a bijection from GCCs to perfect matchings. Then in \S 6 we give the proof of the other main results of \S4. In \S7 we give the bijection between GCSs and broken lines. Then we give some examples in \S8. In the appendix, we give another proof of Theorem \ref{MainThm} using $T$-paths.

\noindent\emph{Acknowledgement.} We are grateful to Ralf Schiffler for valuable correspondences, and to Man Wai Cheung, Mark Gross and Greg Muller for very helpful discussion on scattering diagrams and broken lines. We are grateful to the anonymous referees for carefully reading through the manuscript and giving us many constructive suggestions to improve the presentation.

\section{Background on cluster algebras and type $A$ quivers}
In this section, we recall some definitions and fix notations about quivers and skew-symmetric cluster algebras (\S2.1) and some special type $A$ quivers (\S2.2).

\subsection{Quivers and skew-symmetric cluster algebras}\label{subsection:quiver def}
Recall that a finite oriented graph is a quadruple $Q=(Q_0,Q_1,h,t)$ formed by a finite set of vertices $Q_0$, a finite set of arrows $Q_1$ and two maps $h$ and $t$ from $Q_1$ to $Q_0$ which send an arrow $\alpha$ respectively to its head $h(\alpha)$ and its tail $t(\alpha)$. An arrow $\alpha$ whose head and tail coincide is a \emph{loop}; a \emph{$2$-cycle} is a pair of distinct arrows $\beta$ and $\gamma$ such that $h(\beta)=t(\gamma)$ and $t(\beta)=h(\gamma)$. Similarly, it is clear how to define $n$-cycles for $n\ge3$. A vertex is a \emph{source} (respectively a \emph{sink}) if it is not the head (resp.~the tail) of any arrow.

In this paper, a \emph{quiver} is a finite oriented graph without loops or 2-cycles.

Given a quiver $Q$ and a vertex $v\in Q_0$, the mutation $\mu_v(Q)$ is the new quiver $Q'$ obtained as follows:
\vspace{-5pt}
\begin{enumerate} \itemsep=5pt
\item For every path of the form $u\to v\to w$, add a new arrow from $u$ to $w$.
\item Reverse all arrows incident to $v$.
\item Remove all 2-cycles.
\end{enumerate}

\smallskip

Let $Q=(Q_0,Q_1,h,t)$ be a quiver. Let $Q_0=\{v_1,\ldots,v_n\}=\{1,\dots,n\}$ (for simplicity, we denote $v_i$ by $i$ in this paper if no confusion arises; and later we also use notation $I=I_{\rm uf}=Q_0$ to denote the same set). Let $F=\Q(x_1,\ldots,x_n)$ be the field of rational functions in $x_1,x_2,\ldots,x_n$ with rational coefficients.  A \emph{seed} is a pair $({\bf u}, Q)$ where $u=\{u_1,u_2,\ldots,u_n\}$ is a set of elements of $F$ which freely generate the field $F$.
For any vertex $i\in Q_0$, we denote
$$\prod_{j\to i}u_j=\prod_{\alpha\in Q_1,h(\alpha)=i}u_{t(\alpha)},\quad\quad
 \prod_{i\to j}u_j=\prod_{\alpha\in Q_1,t(\alpha)=i}u_{h(\alpha)}.$$
The mutation $\mu_i({\bf u},Q)$ is the seed $({\bf u}', Q')$  where $Q'=\mu_i(Q)$ and ${\bf u}'$ is obtained from ${\bf u}$ by replacing $u_i$ by
$$u'_i=\dfrac{\dis\prod_{j\to i}u_j+\prod_{i\to j}u_j}{u_i}.$$
Let $(\{x_1\dots,x_n\},Q)$ be the initial seed.  A \emph{cluster} is a set ${\bf u}'$ which appears in a seed $({\bf u}', Q')$ obtained from the initial seed by iterated mutations. An element in a cluster is called a \emph{cluster variable}. A \emph{cluster monomial} is a product of cluster variables in the same cluster. The (coefficient-free) \emph{cluster algebra} $\A(Q)$ associated with $Q$ is the subring of  $F$ generated by all cluster variables.

Next we recall the definition of the cluster algebra $\mathcal{A}_{\rm prin}$ with principal coefficient corresponding to the coefficient-free cluster algebra $\A=\A(Q)$. Let $I_{\rm uf}=Q_0=\{1,\dots,n\}$, $I=\{1,\dots,2n\}$. Define a quiver $\tilde{Q}$ with the vertex set $\tilde{Q}_0:=I$ and the edge set
$$\tilde{Q}_1:=Q_1\cup \{(n+i,i)|i=1,\dots,n\}.$$
In other words, $\tilde{Q}$ is obtained from $Q$ by adding an arrow from $n+i$ to $i$ for each $i=1,\dots,n$. We call $i\in I_{\rm uf}$ unfrozen vertices and $i\in I\setminus I_{\rm uf}$ frozen vertices. Starting with the initial seed $(\{x_1,\dots,x_{2n}\},\tilde{Q})$, we mutate it iteratively similar as above, with the restriction that we only use the mutations $\mu_i$ for $1\le i\le n$ (that is, only mutate at the unfrozen vertices). For each new seed $(\{x_1',\dots,x_{2n}'\},\tilde{Q}')$, the first $n$ rational functions $x_1',\dots,x_n'$ form a cluster. (Note that the frozen variables do not change after mutation, that is, $x_i'=x_i$ for $n+1\le i\le 2n$; and we do not consider them to be cluster variables.) The union of all clusters gives the set of cluster variables.   The cluster algebra $\mathcal{A}_{\rm prin}$ is the subring of $\mathbb{Q}(x_1,\dots,x_{2n})$ generated over  $\mathbb{QP}:=\mathbb{Q}[x_{n+1}^{\pm1},\dots,x_{2n}^{\pm1}]$ by all cluster variables.

\smallskip

Let $Q'=(Q_0', Q_1', h', t')$ be another quiver. We say that $Q$ is a \emph{subquiver} of $Q'$ if
$$Q_0\subseteq Q'_0 \text{ and } Q_1\subseteq Q'_1$$
and $h(e)=h'(e)$ and $t(e)=t'(e)$ for any arrow $e\in Q_1$.
We say that $Q$ is a \emph{full subquiver} of $Q'$ if $Q$ can be obtained from $Q'$ by removing vertices $Q'_0\setminus Q_0$ and their incident arrows.


\subsection{Special classes of type $A$ quivers}
Here we define type $A$, linear,  completely extended linear, and extended linear quivers. The relation among the four classes can be described as follows:
$$
\aligned
&\{\textrm{type }A\} \supset  \{\textrm{extended linear}\}\supset \{\textrm{completely extended linear}\}\\
&\hspace{13em} \raisebox{.6em}{\rotatebox[origin=c]{-45}{\large $\supset$}}
\, \{\textrm{linear}\}
\endaligned
$$
\subsubsection{Type A quivers}

By definition, type $A$ quivers are those that are mutation equivalent to quivers 
of the form $\bullet\to \bullet\to \cdots\to \bullet$. In \cite{BV}, A. Buan and D. Vatne showed that a type $A$ quiver is a connected quiver such that
\begin{itemize}[leftmargin=20pt, itemsep=5pt]
\item all nontrivial simple cycles in the underlying graph have length 3, and the corresponding directed subgraphs are oriented (3-cycles);
\item the vertex degrees of the underlying graph are at most 4; moreover, a degree-4 vertex belongs to two 3-cycles, a degree-3 vertex belongs to one 3-cycle.
\end{itemize}

\subsubsection{Linear quivers}\label{subsection:Linear quivers}
For two integers $a$ and $b$, we denote $[a,b]=\{a,a+1,\ldots,b\}$ if $a\le b$, and $[a,b]=\emptyset$ if $a>b$.

A \emph{linear quiver} is a quiver with $n$ vertices $\{v_1,v_2,\ldots,v_n\}$ and $n-1$ arrows in which any two consecutive vertices $v_i$ and $v_{i+1}$ $(i\in[1,n-1])$ are connected by a single arrow in either direction and there are no others arrows.

In order to have a convenient description for a linear quiver $Q$, we construct a sequence $\{\delta_i\}_{1\leq i \leq n-1}$ such that  $\delta_i=0$ if there is an arrow going from the vertex $v_i$ to the vertex $v_{i+1}$, and $\delta_i=1$ otherwise.
For example, for the quiver $1\to 2\to 3\to 4\leftarrow 5$, we have $(\delta_1,\dots,\delta_4)=(0,0,0,1)$.


\subsubsection{Completely extended linear quiver} We define a \emph{completely extended linear quiver} $Q'$ as obtained from a linear quiver $Q$ by attaching a 3-cycle to every edge,
 a 3-cycle to $v_1$, and a 3-cycle to $v_n$.
(So $Q'$ has $2\lqv+3$ vertices.) By abuse of terminology,  we also call the pair $(Q,Q')$ a completely extended linear quiver, whenever we need to specify the linear quiver $Q$.

Let $(Q,Q')$ be a completely extended linear quiver with $Q_0=\{v_1,\ldots,v_{\lqv}\}$ and $Q'_0 =\{v_1, \dots, v_{2n+3}\}$. If $v_i\in Q'_0\setminus Q_0$ is adjacent to both $v_j$ and $v_{j+1}$ then we also denote $v_{j,j+1}=v_i$. For the 3-cycle attached to $v_1$, the head (resp. tail) of the outgoing (resp. incoming) arrow is denoted $v_{1,0}$ (resp. $v_{1,1}$). We define $v_{n,0}$ and $v_{n,1}$ similarly.

\begin{example} \label{QuiverW3Cycles}
The quiver $(Q, Q')$ in Figure \ref{CELNQuiver} is a completely extended linear quiver, where $Q$ is the linear part $1 \to 2 \to 3 \leftarrow 4$. In there, $v_{1,0}:=5$, $v_{1,1}:=6$, $v_{1,2}:=7$, $v_{2,3}:=8$, $v_{3,4}:=9$, $v_{4,0}:=10$ and $v_{4,1}:=11$.
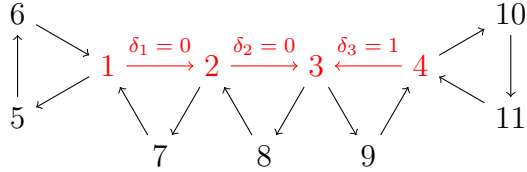
\begin{figure}[h!]
\begin{tikzpicture}[scale=0.8]
\node (v1) at (0:1) [red] {1};
\node (v2) at (0:2.732) [red] {2}; 
\node (v3) at (0:4.464) [red] {3}; 
\node (v4) at (0:6.196) [red] {4};
\node (v5) at (-120:1) {5};
\node (v6) at (120:1) {6};
\path (-60:1.732)++(1,0) node (v7) {7};
\path (-60:1.732)++(2.732,0) node (v8) {8};
\path (-60:1.732)++(4.464,0) node (v9) {9};
\path (30:1.732)++(6.196,0) node (v10) {10};
\path (-30:1.732)++(6.196,0) node (v11) {11};
		
\draw [->] (v1) to (v5);
\draw [->] (v5) to (v6);
\draw [->] (v6) to (v1);
\draw [->,red] (v1) -- (v2) node[midway,above,font=\tiny] {$\delta_1=0$};
\draw [->] (v2) to (v7);
\draw [->] (v7) to (v1);
\draw [->,red] (v2) -- (v3) node[midway,above,font=\tiny] {$\delta_2=0$};
\draw [->] (v3) to (v8);
\draw [->] (v8) to (v2);
\draw [->,red] (v4) -- (v3) node[midway,above,font=\tiny] {$\delta_3=1$};
\draw [->] (v3) to (v9);
\draw [->] (v9) to (v4);
\draw [->] (v4) to (v10);
\draw [->] (v10) to (v11);
\draw [->] (v11) to (v4);
\end{tikzpicture}
\caption{A completely extended linear quiver}\label{CELNQuiver}
\end{figure}
\end{example}

For convenience, if $v_{j,k}=v_i$, then we denote the variable $x_{j,k}=x_i$.

\subsubsection{Extended linear quivers}
An \emph{extended linear quiver} $(Q,P)$ is obtained from a completely extended linear quiver $(Q,Q')$ by removing some vertices (or none) in $Q'_0\setminus Q_0$ and the arrows incident with them. Equivalently, we can characterize $P$ as a quiver obtained from $Q$ by adding some (or none) of the following:
\begin{itemize}[leftmargin=20pt, itemsep=5pt]
\item a 3-cycle or an edge hung on $v_1$, or
\item a 3-cycle or an edge hung on $v_n$, or
\item 3-cycles attached to some edges of $Q$.
\end{itemize}

There is an obvious way to obtain a completely extended linear quiver $(Q,Q')$ from an extended linear quiver $(Q,P)$ (up to relabeling vertices in $Q'_0\setminus P_0$). An example is shown in Figure \ref{completeAQuiver}.
\begin{figure}[h!]
\begin{tabular}{lll}
\begin{tikzpicture}[scale=0.8]
\node (v1) at (0:1) [red] {1};
\node (v2) at (0:2.732) [red] {2}; 
\node (v3) at (0:4.464) [red] {3}; 
\node (v4) at (-120:1) {4};
\path (-60:1.732)++(2.732,0) node (v5) {5};

\draw [->] (v1) to (v4);
\draw [red,->] (v2) to (v1);
\draw [red,->] (v2) to (v3);
\draw [->] (v3) to (v5);
\draw [->] (v5) to (v2);
\end{tikzpicture}
&
\begin{tikzpicture}[scale=0.8]
\node (v1) at (0:0) {};
\node (v2) at (0:2) {};
\node (v3) at (-60:1.732) {};

\draw [->] (v1) to (v2);

\end{tikzpicture}
&
\begin{tikzpicture}[scale=0.8]
\node (v1) at (0:1) [red] {1};
\node (v2) at (0:2.732) [red] {2}; 
\node (v3) at (0:4.464) [red] {3}; 
\node (v4) at (-120:1) {4};
\path (-60:1.732)++(2.732,0) node (v5) {5};
\node (v6) at (120:1) [blue] {6};
\path (-60:1.732)++(1,0) node (v7) [blue] {7};
\path (30:1.732)++(4.464,0) node (v8) [blue] {8};
\path (-30:1.732)++(4.464,0) node (v9) [blue] {9};

\draw [->] (v1) to (v4);
\draw [->,blue] (v4) to (v6);
\draw [->,blue] (v6) to (v1);
\draw [red,->] (v2) to (v1);
\draw [->,blue] (v1) to (v7);
\draw [->,blue] (v7) to (v2);
\draw [red,->] (v2) to (v3);
\draw [->] (v3) to (v5);
\draw [->] (v5) to (v2);
\draw [->,blue] (v3) to (v8);
\draw [->,blue] (v8) to (v9);
\draw [->,blue] (v9) to (v3);
\end{tikzpicture}
\end{tabular}
\caption{An extended linear quiver before and after being completed}
\label{completeAQuiver}
\end{figure}
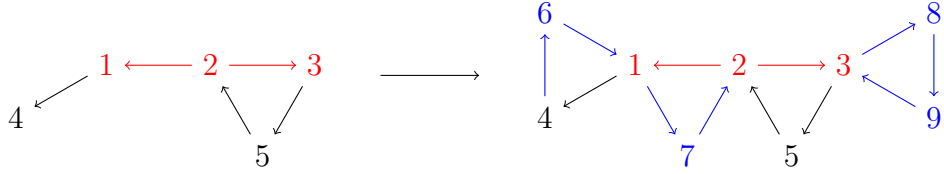

\section{Parametrization of Cluster Monomials by \textbf{d}-vectors}\label{dvecpipe}
\subsection{Cluster monomials  and \textbf{d}-vectors}
It is well known that any cluster algebra associated to a type $A$ quiver with $n$ vertices can be constructed from  a triangulation on a $(n+3)$-gon $P$. 
A \emph{diagonal} on the $(n+3)$-gon $P$ is a line segment connecting two non-adjacent
vertices. 
Two diagonals are said to be crossing if they intersect in the interior of $P$. A triangulation is a maximal set of non-crossing diagonals together with the boundary edges of $P$. 
A connected curve on the polygon is called a \emph{pseudo-diagonal} if it is isotopic to a diagonal (and its endpoints are the same as those of the diagonal) and if its interior is in the interior of the polygon.     Two pseudo-diagonals are said to be \emph{crossing} if they intersect in the interior of the polygon. 
Note that if two pseudo-diagonals are not crossing, then the corresponding diagonals either coincide or  do not cross in the interior of the polygon.

Given a triangulation of $P$, assume the non-crossing diagonals are $T_1,\dots,T_n$ and the boundary edges are $T_{n+1},\dots,T_{2n+3}$. Then it induces a type $A$ quiver $Q$ with $n$ vertices labeled by $T_1,\dots,T_n$, such that there is an arrow from $T_i$ to $T_j$ if and only if there is a triangle with $T_i$ and $T_j$ being its two sides, and $T_j$ is obtained from $T_i$ by rotating ($<180^\circ$) counterclockwisely about their common endpoint. (For example, for the triangulation in Figure \ref{fig:triangulation example}, the corresponding quiver is $T_1\to T_2\to T_3\leftarrow T_4$.) The triangulation also induces a larger quiver $Q'$ with $2n+3$ vertices labeled by $T_1,\dots, T_{2n+3}$ following the same rule as above. We call $T_{n+1},\dots,T_{2n+3}$ the frozen vertices of $Q'$.

Take $P$ and  $Q$ as above. 
It is also known that the cluster variables of $\mathcal{A}(Q)$ are in natural bijection with all the diagonals of $P$.
Using this bijection, a cluster monomial of $\mathcal{A}(Q)$ can be identified with a finite set of pairwise non-crossing (and non-identical) pseudo-diagonals, or equivalently  with 
$$
\{(D_1,d_1),...,(D_m,d_m)\},$$where $m$ is a non-negative integer, $D_1,...,D_m$ are pairwise non-crossing diagonals, and $d_1,...,d_m$ are positive integers. It is easy to see that given $\{(D_1,d_1),...,(D_m,d_m)\}$, we can always find 
 a finite set of pairwise non-crossing (and non-identical) pseudo-diagonals, such that $d_i$ of these pseudo-diagonals are isotopic to $D_i$ for $1\le i\le m$.
 
The following definition uses a natural intersection number, which was already considered in \cite{FST,MS,ST,S2,MSW}.

\begin{definition}
For two diagonals $D,E$ on the $(n+3)$-gon, we define the intersection number $i(D,E)$ of $D$ and $E$ as follows:
$$
i(D,E):=\left\{\begin{array}{ll} 1, & \text{ if }D\text{ and }E\text{ cross each other};\\
-1,  & \text{ if }D\text{ and }E\text{ are the same};\\
0, &\text{ otherwise}.    \end{array}   \right.
$$
Then the \textbf{d}-vector of the cluster monomial $\{(D_1,d_1),...,(D_m,d_m)\}$ is defined by
$$
(a_1,\dots,a_n)=(\sum_{j=1}^m d_j i(D_j,T_1), ..., \sum_{j=1}^m d_j i(D_j,T_n)).
$$
\end{definition}

It is easy to see that the \textbf{d}-vector $(a_1,...,a_n)$ of any cluster monomial satisfies the following:

\noindent \textbf{Property A.} For any 3-cycle $i\to j\to k\to i$ in $Q$ such that $a_i,a_j,a_k$ are positive and satisfy the triangle inequalities (i.e., the sum of any two numbers is strictly greater than the third), the sum $a_i+a_j+a_k$ is even.

\begin{proof}
Assume $i\to j\to k\to i$ is a 3 cycle in $Q$ such that $a_i,a_j,a_k$ are positive and satisfy the triangle inequalities.
Correspondingly, the three diagonals $T_i, T_j, T_k$ form a triangle. Since $a_i\ge 0$, $T_i$ is not in $\{D_1,\dots,D_m\}$. Similar conclusion holds for $T_j$ and $T_k$. There are two cases to consider:

\begin{figure}[h!]
  \centering
    \includegraphics[width=.5\textwidth]{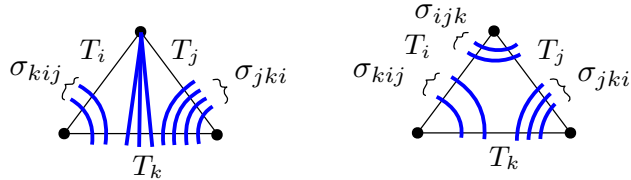}
    \caption{Two cases.}
    \label{pipedream4}
\end{figure}

Case 1: there is at least one diagonal in $\{D_1,\dots,D_m\}$ that crosses an edge and shares a vertex with the triangle formed by $T_i, T_j, T_k$. Without loss of generality, assume $D_\ell$ crosses $T_k$ and shares a vertex with $T_i$ and $T_j$ (see Figure \ref{pipedream4} Left, where we draw pseudo-diagonals to illustrate the multiplicity). In this case all the $a_i$ diagonals (with multiplicity) crossing $T_i$ must cross $T_k$ (otherwise they will cross $D_\ell$). Similarly, all the $a_j$ diagonals (with multiplicity) crossing $T_j$ must cross $T_k$. Thus the total number of diagonals crossing $T_k$ is $a_k=a_i+a_j+a_\ell$, contradicting the  triangle inequality assumption. 

Case 2: each diagonal in  $\{D_1,\dots,D_m\}$ either cross two of $T_i, T_j, T_k$, or none (see Figure \ref{pipedream4} Right). In this case, assume $\sigma_{ijk}$ is the number of diagonals (with multiplicity) crossing $T_i$ and $T_j$, etc. (So $\sigma_{ijk}=\sigma_{jik}$.) Then $a_i=\sigma_{ijk}+\sigma_{kij}$, $a_j=\sigma_{ijk}+\sigma_{jki}$, $a_k=\sigma_{jki}+\sigma_{kij}$, therefore $a_i+a_j+a_k=2(\sigma_{ijk}+\sigma_{jki}+\sigma_{kij})$ is even.
\end{proof}

\subsection{Construction of pipelines}\label{constpipelines}
In this subsection,  the notion of pipelines will be introduced (which we do not claim any originality\footnote{A set of pipelines (to be defined in Proposition \ref{d-vector criteria}) are, in many aspects, similar to a lamination as in \cite{FG}, and a pipeline is similar to a curve in a lamination (we thank one of the referees to pointing it out to us). For example, our construction of the set of pipelines given in Proposition \ref{d-vector criteria} is essentially the same as the ``Reconstruction'' of the lamination given in \cite[page 11]{FG}. The main difference is the following: the two ends of our pipelines are marked points, and that two pipelines may meet at marked points; in contrast, two curves of a lamination do not end at marked points, and two curves do not meet at endpoints. Moreover, we do not  discuss homotopy equivalence of pipelines in this paper.
}) to help us understanding $d$-vectors of cluster monomials.

Let $\mathcal{W}$ be the set of all integer vectors $(a_1,...,a_n)$ satisfying Property A,
 that is, $a_i+a_j+a_k$ is even when both of the following conditions are satisfied:

(i) $i\to j\to k\to i$ is a 3-cycle, and 

(ii) $a_i,a_j,a_k$ are positive and satisfy the triangle inequalities.

\noindent (In particular, if $Q$ has no 3-cycles, then $\mathcal{W}=\mathbb{Z}^n$.)

In this subsection we prove that the cluster monomials of $\mathcal{A}(Q)$ are in bijection with $\mathcal{W}$. We will define a map from $\mathcal{W}$ to the cluster monomials, which would then induce the  immediate bijection. Let $[x]_+=\max(x,0)$ for any real number $x$.
Let ${\bf a}=(a_1,...,a_n)\in\mathcal{W}$.
We define a function $\sigma:\mathbb{R}_{\ge0}^3\to\mathbb{R}_{\ge0}$ as follows:
\begin{equation}\label{function sigma}
\sigma(x,y,z)=\frac{[x+y-z]_+-[x-y-z]_+ -[y-x-z]_+}{2}=\begin{cases}
x, &\textrm{ if }y>x+z;\\
y, &\textrm{ if }x>y+z;\\
0, &\textrm{ if }z>x+y;\\
\displaystyle\frac{x+y-z}{2}, &\textrm{ otherwise}.
\end{cases}
\end{equation}
For convenience, we denote $$\sigma_{ijk}^{\bf a}=\sigma(a_i,a_j,a_k).$$
(If no confusion shall arise, we denote $\sigma_{ijk}=\sigma_{ijk}^{\bf a}$. Note that it coincides with $\sigma_{ijk}$ appeared in the proof of Property A, Case 2.)

\begin{remark}
By case-by-case analysis for the cases shown in Figure \ref{pipedream4}, it is obvious that the equality $\sigma_{kij}+\sigma_{jki}\le [a_k]_+$ (and similar inequalities obtained by permuting $i,j,k$) holds for any 3-cycle $i\to j\to k\to i$ .  
\end{remark}

\begin{proposition}\label{d-vector criteria}
Let $Q$ be a type $A$ quiver. Then
\begin{enumerate}[leftmargin=*, itemsep=5pt]
\item ${\bf a}\in \mathbb{Z}^n$ is the \textbf{d}-vector of some cluster monomial of $\mathcal{A}(Q)$ if and only if ${\bf a}\in\mathcal{W}$.

\item Two distinct cluster monomials have different \textbf{d}-vectors.
\end{enumerate}
\end{proposition}

\begin{proof}
Let $M$ be the set of monomials of $\mathcal{A}(Q)$. An element in $M$ can be identified with a set $\{(D_1,d_1),...,(D_m,d_m)\}$. Let $f: M\to \mathcal{W}$ be the map that sends a cluster monomial to its \textbf{d}-vector. It follows from Property A that the image of $f$ is indeed in $\mathcal{W}$. 

Next we construct a map $g: \mathcal{W}\to M$. Given ${\bf a}\in\mathcal{W}$, we construct  \emph{the set of pipelines associated to} $\bf a$ in three steps. We consider the triangulation of the $(n+3)$-gon corresponding to $Q$, and $T_i, T_j, T_k$ appeared below are diagonals or boundary edges; in other words, $1\le i,j,k\le 2n+3$. 

\noindent \textbf{Step 1:} If $a_i>0$ then draw $a_i$ marked points on  $T_i$  to separate it into $a_i+1$ segments. If $a_i<0$ then draw $-a_i$ pipes so that these pipes are pairwise non-crossing pseudo-diagonals isotopic to $T_i$ and that they do not cross any other $T_j$ $(j\neq i)$. 

\noindent \textbf{Step 2:} If $T_i$ and $T_j$ are two sides of a triangle with the third side $T_k$, then for $1\le r\le \sigma_{ijk}$, we join the two $r$-th marked points on $T_i$ and $T_j$ (ordering in the increasing distance from the common endpoint of $T_i$ and $T_j$) by a pipe inside the triangle. Draw these pipes so that they are disjoint from each other and from the pipes constructed in Step 1.

\noindent \textbf{Step 3:} Suppose that $T_i, T_j, T_k$ form a triangle. Then for each marked point on $T_k$ that is not connected by a pipe to any marked point on $T_i$ or $T_j$ (see Figure \ref{pipedream4} Right), we draw a pipe from the marked point to the common endpoint of $T_i$ and $T_j$. Draw these pipes inside the triangle in such a way that they are non-crossing with each other and with the pipes constructed in Step 1,2.


A pipeline is a union of pipes connected consecutively through the marked points (but not through the vertices of the $(n+3)$-gon). Then the pipelines are pairwise non-crossing. Since the endpoints of each pipeline are non-adjacent vertices of the polygon, every pipeline is a pseudo-diagonal.    Hence the set of pipelines corresponds to a cluster monomial, which we define as the image $g({\bf a})$.

It follows immediately from the above construction that $g$ is the inverse of $f$. So $f$ is bijective. The surjectivity of $f$ implies (1) and the injectivity of $f$ implies (2).
\end{proof}

The above propostion allows us to denote the (unique) cluster monomial with \textbf{d}-vector $(a_1,\dots,a_n)$ by $x[a_1,\dots,a_n]$ or $x[{\bf a}]$.

 Each pipeline $\Lambda$ corresponds to a 0-1 sequence ${\bf b}={\bf b}_\Lambda=(b_1,\dots,b_n)$ such that \begin{equation}\label{b}
b_i=\begin{cases}
&0,\quad\textrm{ if the pipeline $\Lambda$ is disjoint from  $T_i$;}\\
&1,\quad \textrm{ otherwise. }
\end{cases}
\end{equation}
Note that $\Lambda$ corresponds to a linear full subquiver of $Q$. Two pipelines sharing the same endpoints correspond to the same 0-1 sequence.    Let $S$ be  the multiset of 0-1 sequences corresponding to all the pipelines in the set of pipelines associated to ${\bf a}$. Then
\begin{equation}
x[{\bf a}]=\prod_{{\bf b} \in S} x[{\bf b}].
\end{equation}

\begin{example}\label{ex3332431} 
The first two pictures in Figure  \ref{figure:pipelines} are a type $A$ quiver with 7 vertices and its corresponding triangulation on the 10-gon. For a clearer illustration, the 10-gon is drawn as a concave polygon.  The bottom illustrates the construction of pipelines for ${\bf a}=(3,3,3,2,4,3,1)$.

\begin{figure}[h!]

\begin{tikzpicture}[scale=0.8]

\node (v1) at (150:1.732) {1};
\node (v2) at (0:0) {2};
\node (v3) at (0:1.732) {3};
\path (30:1.732)++(1.732,0) node (v4) {4};
\node (v5) at (-60:1.732) {5};
\path (-120:1.732)++(0.866,-1.5) node (v6) {6};
\path (-60:1.732)++(0.866,-1.5) node (v7) {7};
\path (v5)++(0,-2) node [below] {$Q$};

\draw [->] (v2) to (v1);
\draw [->] (v2) to (v3);
\draw [->] (v3) to (v5);
\draw [->] (v5) to (v2);
\draw [->] (v5) to (v6);
\draw [->] (v6) to (v7);
\draw [->] (v7) to (v5);
\draw [->] (v3) to (v4);

\begin{scope}[shift={(5,-1.5)}, inner sep=2pt]
\tikzstyle{every node}=[font=\footnotesize]

\node (v1) at (0:0) {};
\node (v2) at (60:1.732) {};
\node (v3) at (0:1.732) {};
\node (v4) at (-60:1.732) {};

\draw [fill=black] (v1) circle [radius=2pt];
\draw [fill=black] (v2) circle [radius=2pt];
\draw [fill=black] (v3) circle [radius=2pt];
\draw [fill=black] (v4) circle [radius=2pt];

\draw [-] (v1.center) to (v2.center);
\draw [-] (v2.center) -- (v3.center) node[midway,left] {1};
\draw [-] (v3.center) to (v1.center);
\draw [-] (v3.center) to (v4.center);

\path (v2)++(1.732,0) node (v5) {};
\path (v3)++(1.732,0) node (v6) {};
\path (v4)++(1.732,0) node (v7) {};

\draw [fill=black] (v5) circle [radius=2pt];
\draw [fill=black] (v6) circle [radius=2pt];
\draw [fill=black] (v7) circle [radius=2pt];

\draw [-] (v5.center) to (v2.center);
\draw [-] (v7.center) to (v4.center);

\draw [-] (v3.center) -- (v5.center) node[midway,left] {2};
\draw [-] (v5.center) -- (v6.center) node[midway,right] {3};
\draw [-] (v6.center) -- (v3.center) node[midway,above] {5};
\draw [-] (v6.center) -- (v7.center) node[midway,right] {6};
\draw [-] (v7.center) -- (v3.center) node[midway,left] {7};

\path (v2)++(2.598,1.5) node (v8) {};
\path (v3)++(2.598,1.5) node (v9) {};

\draw [fill=black] (v8) circle [radius=2pt];
\draw [fill=black] (v9) circle [radius=2pt];

\draw [-] (v5.center) to (v8.center);
\draw [-] (v8.center) to (v9.center);
\draw [-] (v9.center) -- (v5.center) node[midway,above] {4};
\draw [-] (v9.center) to (v6.center);

\path (v4)++(3.464,0) node (v10) {};

\draw [fill=black] (v10) circle [radius=2pt];

\draw [-] (v6.center) to (v10.center);
\draw [-] (v10.center) to (v7.center);


\path (v7)++(0,-0.5) node [below] {$G$};
\end{scope}
\end{tikzpicture}
\end{figure}

\begin{figure}[h!]
  \centering
    \includegraphics[width=.8\textwidth]{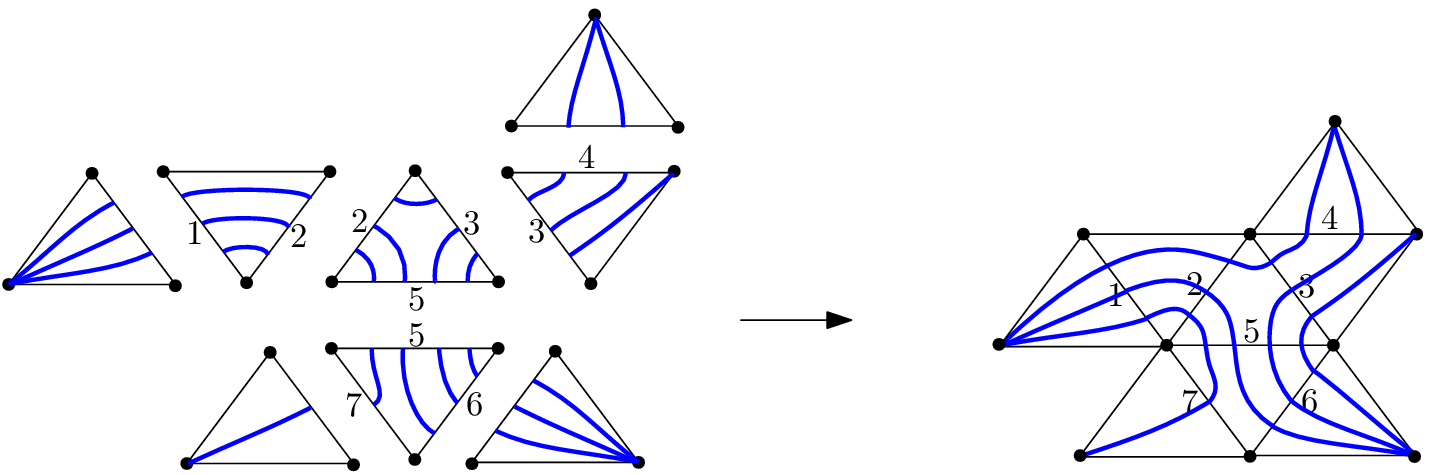}
\caption{The construction of pipelines}
\label{figure:pipelines}
\end{figure}

The set of pipelines associated to ${\bf a}$ consists of 5 pipelines, namely those pass through edges sets $\{1,2,3,4\}$, $\{1,2,5,6\}$, $\{1,2,5,7\}$, $\{3,4,5,6\}$, $\{3,5,6\}$, respectively. So the cluster monomial $x[{\bf a}]$ is decomposed as  $x[3,3,3,2,4,3,1]=x[1,1,1,1,0,0,0] \cdot x[1,1,0,0,1,1,0] \cdot x[1,1,0,0,1,0,1] \cdot x[0,0,1,1,1,1,0] \cdot x[0,0,1,0,1,1,0]$, and the multiset $S$ (which is indeed a set in this example) consists of the five 0-1 sequences appeared above.
\end{example}

\begin{lemma}\label{a to a+}
For ${\bf a}\in \mathbb{Z}^n$, assume that $[{\bf a}]_+:=[[a_1]_+,\dots,[a_n]_+]$ satisfies Property A and $x[[{\bf a}]_+]=\prod x[{\bf b}]$ is a factorization into cluster variables in the same cluster. Then $$x[{\bf a}]=x[[{\bf a}]_+]\prod_i x_i^{[-a_i]_+}=\prod x[{\bf b}]\prod_i x_i^{[-a_i]_+},$$
and that $x[{\bf b}]$ and $x_i$ ($a_i<0$) are in the same cluster.
\end{lemma}
\begin{proof}
If $a_i<0$, then there is no ${\bf b}$ satisfying $b_i>0$. Thus no pseudo-diagonals corresponding to pipelines for $[{\bf a}]_+$ will cross the diagonal $T_i$. Therefore,  after adding  $T_i$ we still get a set of non-crossing pseudo-diagonals, which means that $x[{\bf b}]$ and $x_i$ ($a_i<0$) are in the same cluster.
It then follows that $x[{\bf a}]=x[[{\bf a}]_+]\prod_i x_i^{[-a_i]_+}$.
\end{proof}

\begin{remark}\label{rmk2.1}
Assume that $Q$ is a full subquiver of $Q'$. We compare cluster variables in various cluster algebras.

For simplicity, we denote the vertex sets $Q_0=\{1,\dots,n\}$ and $Q'_0=\{1,\dots,n'\}$. By definition, the cluster algebra with coefficients, denoted $\mathcal{A}(Q,Q')$, is generated by only those cluster variables in $\mathcal{A}(Q')$ obtained from iteratively mutating the initial cluster variables $x_1,\dots,x_n$ only at vertices in $\{1,\dots,n\}$. (Vertices in $\{n+1,\dots,n'\}$ are called frozen vertices. The coefficients are in $\mathbb{Z}[x_{n+1}^{\pm1},\dots,x_{n'}^{\pm1}]$.) Thus, there is a natural bijection sending a cluster variable in $\mathcal{A}(Q,Q')$ of the form $x[d_1,\dots,d_n]$ to the cluster variable $x[d_1,\dots,d_n,0,\dots,0]\in \mathcal{A}(Q')$ of the same expression.

There is also a natural bijection sending a cluster variable $x[d_1,\dots,d_n]\in\mathcal{A}(Q,Q')$ to the cluster variable $x[d_1,\dots,d_n]\in\mathcal{A}(Q)$, given by substituting $x_i$ by $1$ for $i\in [n+1,n']$. More generally, if $Q$ is a full subquiver of $P$, and $P$ is a full subquiver of $Q'$, then there is a natural bijection sending a cluster variable $x[d_1,\dots,d_n]\in\mathcal{A}(Q,Q')$ to $x[d_1,\dots,d_n]\in\mathcal{A}(Q,P)$ given by substituting $x_i$ by $1$ for $i\in Q'_0\setminus P_0$.

If $Q$ is a full subquiver of $Q''$, and $Q'$ is the vertex-induced subquiver of $Q''$ whose vertex set consists of vertices in $Q_0$ and those adjacent to $Q_0$, then the cluster variables in $\mathcal{A}(Q,Q')$ have the same expressions as those in $\mathcal{A}(Q,Q'')$.
\end{remark}

\section{Globally Compatible Collections} \label{GCCs}

In this section, we give three formulas for the cluster monomial $x[{\bf a}]$ for ${\bf a}\in\mathcal{W}$. All results here will be proved in \S6. 
These formulas extend the results on the equioriented type $A$ quivers given in \cite{B9}. For general type $A$ quivers, a new situation we need to handle is the appearance of 3-cycles.

First we reduce to a special case. We can replace ${\bf a}$ by $[{\bf a}]_+$, thus can assume $a_i\ge 0$, because of Lemma \ref{a to a+}. Moreover, for any edge $i\to j$ of $Q$ that is not in a 3-cycle, we can add a vertex $k$ and two arrows $j\to k$ and $k\to i$ (the vertex $k$ is a frozen vertex). Indeed, assume the modified quiver is $Q'$. Then by Remark \ref{rmk2.1}, once we have a formula for cluster monomials for $Q'$, we can set $x_i=1$ for all $i\in Q'_0\setminus Q_0$ and obtain a formula for cluster monomials for $Q$. In the rest of the paper, we assume that ${\bf a}=[{\bf a}]_+$ and
\begin{equation}\label{assumptionQ}
\textrm{$Q$ is of type $A$ with more than one vertex, and every edge of $Q$ is in a 3-cycle.}
\end{equation}

\subsection{A formula using 0-1 sequences}\label{subsection3.1}
Fix a deg-2 vertex $i_0$ of $Q$ (which exists because of \eqref{assumptionQ}). For $i\in Q_0$,  denote by $d(i)$ be the distance  (i.e., the length of the shortest directed path) from $i_0$ to $i$. Let ${\bf s}_i=(s_{i,1},s_{i,2},\dots,s_{i,a_i})\in\{0,1\}^{a_i}$ be a 0-1 sequence, and define
$$|{\bf s}_i|=\sum_{r=1}^{a_i} s_{i,r}\; ,\quad\quad |\bar{\bf s}_i|=\sum_{r=1}^{a_i} (1-s_{i,r})=a_i-|{\bf s}_i|.$$
We say that a sequence of 0-1 sequences ${\bf s}:=({\bf s}_1,\dots,{\bf s}_n)$ is a \emph{globally compatible sequence (abbreviated GCS)} if  the following holds for any 3-cycle $i\to j\to k\to i$:
\begin{itemize}[leftmargin=20pt, itemsep=5pt]
\item If $d(i)<d(j)<d(k)$, then $(s_{i,t},s_{j,t})\neq (1,0)$  for $1\le t\le \sigma_{ijk}$;
\item If $d(j)<d(k)<d(i)$, then $(s_{i,a_i+1-t},s_{j,a_j+1-t})\neq (1,0)$ for $1\le t\le \sigma_{ijk}$;
\item If $d(k)<d(i)<d(j)$, then $(s_{i,a_i+1-t},s_{j,t})\neq (1,0)$ for $1\le t\le \sigma_{ijk}$.
\end{itemize}

\begin{theorem}\label{thm:01sequence} For any \textbf{d}-vector ${\bf a}=(a_1,\dots,a_n)\in\mathbb{Z}_{\ge0}^n$ (i.e., ${\bf a}\in\mathcal{W}\cap \mathbb{Z}_{\ge0}^n$), we have the following formula for the corresponding cluster monomial:
\begin{equation}\label{formula1}
x[\mathbf{a}]=\left(\prod_{l=1}^n  x_l^{-a_l}\right)
\sum_{\mathbf{s}}\left(\prod_{i}  x_{i}^{e_i}\right),\;\textrm{  where }
e_i=\sum_{i\to j}|\,\bar{\mathbf{s}}_j\,|+\sum_{k\to i} |\mathbf{s}_k|-\sum_{i\to j\to k\to i}\sigma_{jki},
\end{equation}
and ${\bf s}$ runs through all GCSs. 
\end{theorem}

This formula specializes to the formula given in \cite{B9} for a linear quiver.

\subsection{A formula using Dyck paths}\label{subsection3.2}
We recall the following definition from \cite{LLM}.

Let $(a_1, a_2)$ be a pair of nonnegative integers. Let $c=\min(a_1,a_2)$.
The \emph{maximal Dyck path} of type $a_1\times a_2$, denoted by $\mathcal{D}=\mathcal{D}^{a_1\times a_2}$, is a lattice path
from $(0, 0)$  to $(a_1,a_2)$ that is as close as possible to the diagonal joining $(0,0)$ and $(a_1,a_2)$, but never goes above it. A \emph{corner} is a subpath consisting of a horizontal edge followed by a vertical edge.

\begin{definition}\label{corner-first}
Let $\mathcal{D}_1$ (resp.~$\mathcal{D}_2$) be the set of horizontal (resp.~vertical) edges of a maximal Dyck path $\mathcal{D}=\mathcal{D}^{a_1\times a_2}$. We label $\mathcal{D}$ with the \emph{corner-first index} in the following sense:
\vspace{-5pt}
\begin{enumerate}[label=(\alph*), leftmargin=*, itemsep=5pt]
\item edges in $\mathcal{D}_1$ are indexed as $u_1,\dots,u_{a_1}$ such that $u_i$ is the horizontal edge of the $i$-th corner for $i\in [1,c]$ and $u_{c+i}$ is the $i$-th of the remaining horizontal ones for $i\in [1, a_1-c]$,
\item edges in $\mathcal{D}_2$ are indexed as $v_1,\dots, v_{a_2}$ such that $v_i$ is the vertical edge of the $i$-th corner for $i\in [1,c]$ and $v_{c+i}$ is the $i$-th  of the remaining vertical ones for $i\in [1, a_2-c]$.
\end{enumerate}
\noindent (Here we count corners from bottom left to top right, count vertical edges from bottom to top, and count horizontal edges from left to right.)
\end{definition}

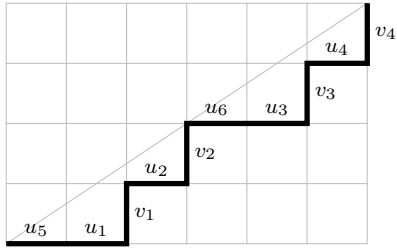
\begin{figure}[h!]
\begin{tikzpicture}[MyPic,scale=0.8]
\tikzstyle{every node}=[font=\tiny]
\tsomepoints

\draw (0,0) grid (6,4);
\draw (0,0) -- (6,4);
\draw [SoThick]
(P00) -- (P10)  node[midway,above] {$u_5$} --
(P10) -- (P20)  node[midway,above] {$u_1$} --
(P20) -- (P21)  node[midway,right] {$v_1$} --
(P21) -- (P31)  node[midway,above] {$u_2$} --
(P31) -- (P32)  node[midway,right] {$v_2$} --
(P32) -- (P42)  node[midway,above] {$u_6$} --
(P42) -- (P52)  node[midway,above] {$u_3$} --
(P52) -- (P53)  node[midway,right] {$v_3$} --
(P53) -- (P63)  node[midway,above] {$u_4$} --
(P63) -- (P64)  node[midway,right] {$v_4$};
\end{tikzpicture}
\caption{A maximal Dyck path}
\end{figure}

\begin{definition}\label{local_compatibility}
Let  $S_1\subseteq \mathcal{D}_1$, $S_2 \subseteq \mathcal{D}_2$, $s\in\mathbb{Z}_{\ge0}$. We say that $S_1$ and $S_2$ are \emph{$s$-compatible} if for every $1\le r\le s$, either $u_r\notin S_1$ or $v_r\notin S_2$. In other words, neither of the first $s$ corners are in the subpath $S_1\cup S_2$.
\end{definition}

Note that the following definition of global compatibility is different from \cite{LLM}.

\begin{definition}\label{definition:GCC}
Let ${\bf a}=(a_1,\dots,a_n)\in\mathbb{Z}_{\ge0}^n$  be a \textbf{d}-vector.
For each pair $i\to j$ in $Q$, let $\mathcal{D}^{(i,j)}$ be a maximal Dyck path $\mathcal{D}^{a_i\times a_j}$. We label  $\mathcal{D}^{(i,j)}$ with the corner-first index (described in Definition \ref{corner-first}), whose horizontal edges are denoted $u_{1}^{(i,j)},\dots,u_{a_{i}}^{(i,j)}$ and vertical edges are denoted by $v_{1}^{(i,j)},\dots,v_{a_{j}}^{(i,j)}$.
We say that the collection
$$\Big\{S_\ell^{(i,j)} \subseteq\mathcal{D}^{(i,j)}_\ell\;\Big|\;  i\to j
\textrm{ is an arrow}, \; \ell\in\{1,2\}\Big\}$$ is a \emph{globally compatible collection (abbreviated GCC)} if and only if for any $k\to i\to j$ in $Q$:
\begin{itemize}[leftmargin=20pt, itemsep=5pt]
\item if $j\to k$ is also an arrow in $Q$, then $S^{(i,j)}_1$ and $S^{(i,j)}_2$ are $\sigma_{ijk}$-compatible,
and
$$v^{(k,i)}_r\in S^{(k,i)}_2\Longleftrightarrow u^{(i,j)}_{a_i+1-r}\not\in S^{(i,j)}_1, \textrm{ for all $r\in[1,a_i]$};$$

\item otherwise,   $$v^{(k,i)}_r\in S^{(k,i)}_2\Longleftrightarrow u^{(i,j)}_r\not\in S^{(i,j)}_1,\textrm{ for all $r\in[1,a_i]$}.$$
\end{itemize}
\end{definition}

\begin{theorem}\label{thm:GCCxa} Assume $n>1$. For any \textbf{d}-vector ${\bf a}=(a_1,\dots,a_n)\in\mathbb{Z}_{\ge0}^n$, we have the following formula for the corresponding cluster monomial:
\begin{equation}\label{formula2}
x[\mathbf{a}]=\left(\prod_{l=1}^n  x_l^{-a_l}\right)
\sum\left(\prod_{i\to j}  x_{i}^{\big|S^{(i,j)}_2\big|} x_{j}^{\big|S^{(i,j)}_1\big|}\right)\cdot \prod_{i\to j\to k\to i} x_i^{-\sigma_{jki}},
\end{equation}
where the sum runs over all GCCs.  (Note that because of the rotational symmetry,  each 3-cycle contributes three terms  to the last product.)
\end{theorem}


\subsection{A more explicit formula for cluster variables}\label{subsection3.3}
The third method of computing the cluster monomial with given \textbf{d}-vector is to  first decompose the cluster monomial into a product of cluster variables using pipelines and then use a formula for cluster variables using GCCs.

For given \textbf{d}-vector, we construct pipelines as in the subsection \ref{constpipelines}. Then each pipeline corresponds to a cluster variable. We give the GCC formula of the cluster variable $x[{\bf a}]$ in a completely extended linear quiver $(Q,Q')$, where ${\bf a}=(a_1,\dots,a_{n'})$ such that $a_i=1$ if $i\in Q$, and $a_i=0$ otherwise.
For any arrow $(i+\delta_i)\to (i+1-\delta_i)$ of $Q$ we attach a Dyck path $\mathcal{D}^{(i)}=\mathcal{D}^{1\times 1}$, which consists of one horizontal edge and one vertical edge. (Recall that $\delta_i$ is defined in \S\ref{subsection:Linear quivers}.) Then Definition \ref{definition:GCC} specializes to the following:

Let $S_{i,r}\subseteq\D^{(i)}_r$ for $i\in[1,n-1]$, $r\in[1,2]$. We say that the collection $\{S_{i,r}\}$ is a \emph{GCC} if
\begin{equation}\label{GCCRule}
(|S_{i,1}|,|S_{i,2}|)\neq(1,1) \textrm{ for } i\in[1,n-1], \textrm{ and the following holds for $i\in [2,n-1]$:}
\end{equation}

\begin{enumerate}[label=(\alph*), itemsep=5pt]
\item if $(\delta_{i-1},\delta_{i})=(0,0)$, then $|S_{i-1,2}|\neq|S_{i,1}|$;
\item if $(\delta_{i-1},\delta_{i})=(1,1)$, then $|S_{i-1,1}|\neq|S_{i,2}|$;
\item if $(\delta_{i-1},\delta_{i})=(0,1)$, then $|S_{i-1,2}|=|S_{i,2}|$;
\item if $(\delta_{i-1},\delta_{i})=(1,0)$, then $|S_{i-1,1}|=|S_{i,1}|$.
\end{enumerate}

\begin{theorem} \label{MainThm}
Let $(Q,Q')$ be a completely extended linear quiver, and $n=|Q_0|>1$, $n'=|Q'_0|$.
Define ${\bf a}_Q=(a_1,\dots,a_{n'})$ such that $a_i=1$ if $i\in Q$, and $a_i=0$ otherwise.
Then the cluster variable with \textbf{d}-vector ${\bf a}_Q$ is
\begin{equation}\label{formula3}
x[\bfa_Q]=\left(\prod_{i=1}^n x_{i}^{-1}\right) \sum \left(\prod_{i=0}^{\lqv} y_i \right),
\end{equation}
where the sum runs over all GCCs $\{S_{i,r}\}$,
$$y_i:=x_{i+\delta_i}^{|S_{i,2}|} x_{i+1-\delta_i}^{|S_{i,1}|}x_{i,i+1}^{1-|S_{i,1}|-|S_{i,2}|}=
\left\{\begin{array}{lll}
x_{i}, & \text{if} & (|S_{i,1}|, |S_{i,2}|)=(\delta_i,1-\delta_i) \xh
x_{i+1}, & \text{if} & (|S_{i,1}|, |S_{i,2}|)=(1-\delta_i,\delta_i) \xh
x_{i,i+1}, & \text{if} & (|S_{i,1}|, |S_{i,2}|)=(0,0) \xh
\end{array}\right.
$$
for $i\in[1, n-1]$, and
$$y_0:=
\begin{cases}
x_{1,0}, \textrm{ if  }|S_{1,1+\delta_1}|=1-\delta_1, \\
x_{1,1}, \textrm{ otherwise,}\\
\end{cases}
\quad
y_n:=
\begin{cases}
x_{n,0}, \textrm{ if }|S_{n-1,2-\delta_{n-1}}|=\delta_{n-1}, \\
x_{n,1}, \textrm{ otherwise}.\\
\end{cases}$$
\end{theorem}
\begin{remark}\label{rmk 4.7}
Note that the above theorem induces a formula for the cluster variables of any type $A$ quiver,  
that is, any mutation-equivalent type $A$ quiver (possibly including some 3-cycles). Indeed, let $\tilde{Q}$ be any type $A$ quiver and $x[{\bf a}]$ be a non-initial cluster variable (the formulas for initial cluster variables are trivial) with \textbf{d}-vector ${\bf a}=(a_1,\dots,a_{n'})$, with  $n'=|\tilde{Q}_0|$. Then the subset of vertices $\{i\, |\, a_i=1\}$ is equal to the set of vertices $Q_0$ of a linear full subquiver $Q$. By relabeling vertices if necessary, we assume $Q_0=\{1,\dots,n\}$ (thus $a_1=\cdots=a_n=1$ and $a_{n+1}=\cdots=a_{n'}=0$ and for convenience we denote ${\bf a}=1^n0^{n'-n}$. By removing vertices in $\tilde{Q}_0$ but not in or adjacent to $Q_0$, we get an extended linear quiver $(Q,P)$; from this extended linear quiver we can obtain a completely extended linear quiver $Q'$. Define $n'=|Q'_0|$, $m=|P_0|$. Thanks to Remark \ref{rmk2.1}, a formula for cluster variable  $x[1^n0^{n'-n}]\in \mathcal{A}(Q,Q')$ induces a formula for $x[1^n0^{m-n}]\in \mathcal{A}(Q,P)$ by substituting $x_i=1$ for $i\in Q'_0\setminus P_0$, which is also a formula for $x[{\bf a}]=x[1^n0^{n'-n}]\in \mathcal{A}(\tilde{Q})$. (See Example \ref{example62}.)
\end{remark}

\begin{remark} We also give a formula for cluster variables of any type $A$ quiver in terms of GCS as follows. Let $\tilde{Q}$ be a type $A$ quiver and $Q$ be a linear full subquiver of $\tilde{Q}$, as in Remark \ref{rmk 4.7}. Define 
\begin{equation}\label{df:a}
{\bf a}_Q=(a_i), \quad  \textrm{ where }
a_i= \left\{\begin{array}{lr}
        1, \quad \text{if $i\in Q_0$;}\\
        0, \quad \text{if $i\notin Q_0$.}\\
        \end{array}\right.
\end{equation}
Relabelling vertices if necessary, we assume that $Q$ is $1\longleftrightarrow 2 \longleftrightarrow\cdots \longleftrightarrow n$ where each arrow can go either direction.
We will give a formula of the cluster variable $x[{\bf a}_Q]$.

Let ${\bf s}=(s_i)\in\{0,1\}^{n'}$, $0\le s_i\le a_i$ for every $1\le i\le n'$.  We say that ${\bf s}$ is a GCS if  $(s_i,s_j)\neq (1,0)$ for every arrow $i\to j$ in $Q_1$. Because $s_i=0$ for $i>n$, by abuse of notation we also view ${\bf s}$ as an element in $\{0,1\}^n$. 

Let $\deg^-_Q(i)$ be the number of 
arrows in $\tilde{Q}_1$ whose tails are in $Q_0$ and heads are the vertex $i$, let $\deg^+_Q(i)$ be the number of 
arrows in $\tilde{Q}_1$ whose heads are in $Q_0$ and tails are the vertex $i$.
Let $K$ be the set of vertices $k\in \tilde{Q}_0\setminus Q_0$ such that there exists a 3-cycle $k\to i\to j\to k$ with $s_i=s_j=1$; or equivalently,
$$K=\{k\in \tilde{Q}_0\setminus Q_0\; | \; \deg^-_Q(k)=\deg^+_Q(k)=1\}.$$
Denote $\bar{s_i}=a_i-s_i$. We have the following formula, where arrows $i\to j$ are in $\tilde{Q}$:
\begin{equation}\label{cluster variable GCS formula}
x[\mathbf{a}_Q]
=\left(\prod_{r\in Q_0}  x_r^{-1}\right)
\frac{\sum_{\mathbf{s}}\left(\prod_{i\to j}  x_i^{\bar{s}_j}x_j^{s_i}\right) }{ \prod_{k\in K}x_k}
=
\sum_{\mathbf{s}} z_{\bf s},\quad \textrm{where } z_{\bf s}=\frac{\prod_{i\to j}  x_i^{\bar{s}_j}x_j^{s_i}}{\prod_{r\in Q_0\cup K}  x_r}
\end{equation}
and ${\bf s}$ runs through all GCSs. 

\end{remark}
\section{A bijection between perfect matchings and GCCs} \label{PerfectMatchingsGCCs}
In this section, we first recall the construction of snake diagram and the formula of cluster variables using perfect matching as in \cite{MS}, then give a bijective proof of Theorem \ref{MainThm} via perfect matchings.

Associated to a completely extended linear quiver $(Q,Q')$, we recursively construct the snake diagram by gluing $n$-tiles together as follows: we first put the $2^{\text{nd}}$-tile to the right side of the $1^{\text{st}}$-tile;  suppose the \th{i}-tile is placed, we add the \th{(i+1)}-tile to the right side or on top of the \th{i}-tile such that the \th{(i-1)}, \th{(i)} and \th{(i+1)}-tiles are in the same row or column if and only if  $\delta_{i-1}\neq\delta_i$.

Next, we label the edges as follows.
\begin{itemize}[itemsep=5pt]
\item The common edge of the \th{i}-tile and the \th{(i+1)}-tile is labeled $T_{i,i+1}$.
\item Denote by $Pl^{(i)}$ the parallelogram bounded by the main diagonals of the \th{i}-tile and the \th{(i+1)}-tile and two boundary edges. Any edge forming an angle of $135^\circ$ with the main diagonal of the \th{i}-tile will be labeled $T_i^{(j)}$ (where $j$ indicates the parallelogram to which the edge belongs).
\item For convenience, we let $Pl^{(0)}$ be the right triangle with legs $T_{1,0}$ and $T_{1,1}$, and let $Pl^{(n)}$ the right triangle with legs $T_{n,0}$ and $T_{n,1}$.
\item The edges of the first and the last tiles are labeled as in Figure \ref{FirstLastTiles}.
\end{itemize}

\begin{figure}[h!]
\begin{tikzpicture}[MyPic,scale=1.7]
\tikzstyle{every node}=[font=\tiny]
\tsomepoints
	
\draw (P00) -- (P10) node[midway,below] {$T_{i-2}^{(i-2)}$};
\draw (P11) -- (P10) node[midway,above,sloped] {$T_{i-1,i}$};
\draw (P11) -- (P01) node[midway,above] {$T_i^{(i-1)}$};
\draw (P01) -- (P00);
\node at (0.5,0.5) [fill=lightgray] {$i-1$};
	
\draw (P10) -- (P20) node[midway,below] {$T_{i-1}^{(i-1)}$};
\draw [very thick, line width=3pt] (P20) -- (P21) node[midway,right] {$T_{i+1}^{(i)}$};
\draw (P21) -- (P11) node[midway,above] {$T_{i,i+1}$};
\draw (P11) -- (P20);
\draw [very thick, line width=3pt] (P11) -- (P20);
\node at (1.5,0.5) [fill=lightgray] {$i$};
	
\draw (P21) -- (P22) node[midway,right] {$T_{i+2}^{(i+1)}$};
\draw (P22) -- (P12);
\draw [very thick, line width=3pt] (P12) -- (P11) node[left,near start] {$T_i^{(i)}$};
\draw [very thick, line width=3pt] (P12) -- (P21);
\node at (1.5,1.5) [fill=lightgray] {$i+1$};

\node [above,font=\normalsize] at (1,-1) {$\delta_{i-1}=\delta_i$};
	
\begin{scope}[shift={(4,0)}]
\tikzstyle{every node}=[font=\tiny]
\tsomepoints
	
\draw (P00) -- (P10);
\draw (P11) -- (P10) node[midway,above,sloped] {$T_{i-1,i}$};
\draw (P11) -- (P01) node[midway,above] {$T_i^{(i-1)}$};
\draw (P01) -- (P00);
\node at (0.5,0.5) [fill=lightgray] {$i-1$};
	
\draw (P10) -- (P20) node[midway,below] {$T_{i-1}^{(i-1)}$};
\draw (P21) -- (P20) node[midway,above,sloped] {$T_{i,i+1}$};
\draw [very thick, line width=3pt] (P21) -- (P11) node[midway,above] {$T_{i+1}^{(i)}$};
\draw [very thick, line width=3pt] (P11) -- (P20);
\node at (1.5,0.5) [fill=lightgray] {$i$};
	
\draw [very thick, line width=3pt] (P20) -- (P30) node[midway,below] {$T_i^{(i)}$};
\draw (P30) -- (P31);
\draw (P31) -- (P21);
\draw [very thick, line width=3pt] (P21) -- (P30);
\node at (2.5,0.5) [fill=lightgray] {$i+1$};
	
\node [above,font=\normalsize] at (1.5,-1) {$\delta_{i-1}\ne \delta_i$};
\end{scope}
\end{tikzpicture}
\caption{Labels of edges in \th{(i-1)}, \th{i} and \th{(i+1)}-tiles in two cases}
\label{labels}
\end{figure}
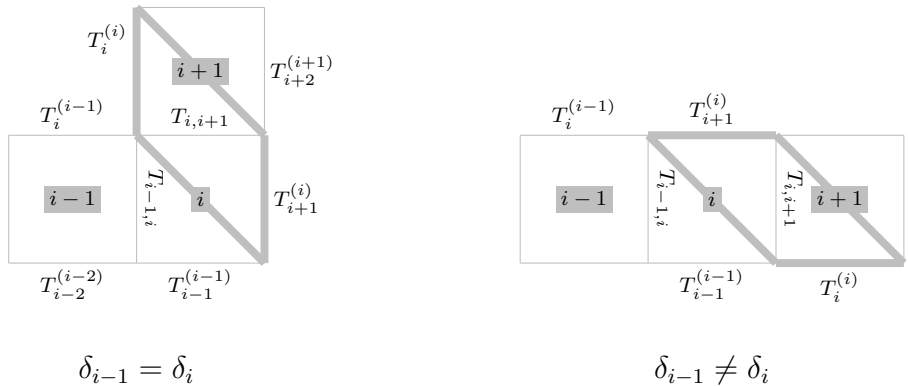

\begin{figure}[h!]
	\begin{tikzpicture}[MyPic,scale=1.4]
	\tikzstyle{every node}=[font=\small]
	\tsomepoints
	
	\draw (P00) -- (P10) node[midway,below=2pt] {$T_{1,\delta_1}$};
	\draw (P10) -- (P11) node[midway,right] {$T_{1,2}$};
	\draw (P11) -- (P01) node[midway,above] {$T_2^{(1)}$};
	\draw (P01) -- (P00) node[midway,left] {$T_{1,1-\delta_1}$};
	\node at (0.5,0.5) [fill=lightgray] {$1$};
	\node [font=\normalsize] at (0.5,-1) {First tile};
	\node [font=\normalsize] at (6.5,-1) {Last tile};
	\node [font=\normalsize] at (6.5,0.5) {or};
	
	\begin{scope}[shift={(4,0)}]
	\tikzstyle{every node}=[font=\small]
	\tsomepoints
	
	\draw (P00) -- (P10) node[midway,below=2pt] {$T_{\lqv-1}^{(\lqv-1)}$};
	\draw (P10) -- (P11) node[midway,right] {$T_{\lqv,\delta_{n-1}}$};
	\draw (P11) -- (P01) node[midway,above] {$T_{\lqv,1-\delta_{n-1}}$};
	\draw (P01) -- (P00) node[midway,left] {$T_{\lqv-1,\lqv}$};
	\node at (0.5,0.5) [fill=lightgray] {$\lqv$};
	\end{scope}
	
	\begin{scope}[shift={(8,0)}]
	\tikzstyle{every node}=[font=\small]
	\tsomepoints
	
	\draw (P00) -- (P10) node[midway,below=2pt] {$T_{\lqv-1,\lqv}$};
	\draw (P10) -- (P11) node[midway,right] {$T_{\lqv,1-\delta_{n-1}}$};
	\draw (P11) -- (P01) node[midway,above] {$T_{\lqv,\delta_{n-1}}$};
	\draw (P01) -- (P00) node[midway,left] {$T_{\lqv-1}^{(\lqv-1)}$};
	\node at (0.5,0.5) [fill=lightgray] {$\lqv$};
	\end{scope}
	\end{tikzpicture}
	\caption{Labels of edges in the first and last tiles}
	\label{FirstLastTiles}
\end{figure}
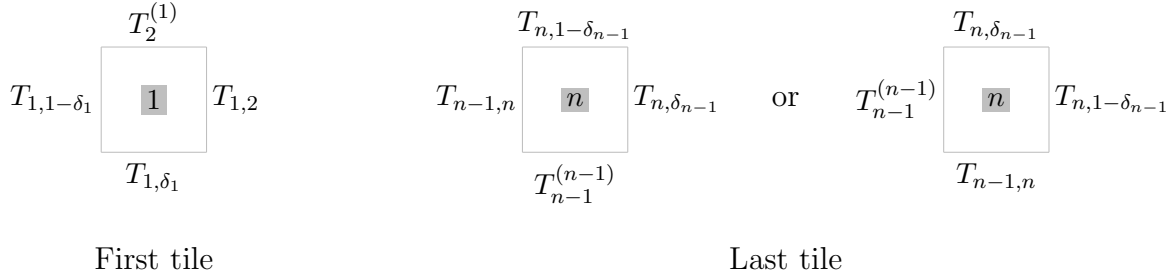

A perfect matching of the snake diagram is a set of edges such that each vertex is incident to exactly one edge in the set.

\begin{figure}[h!]
\begin{tikzpicture}[MyPic,scale=1.4]
\tikzstyle{every node}=[font=\tiny]
\tsomepoints

\draw (P00) -- (P10);
\draw [SoThick] (P10) -- (P11);
\draw (P11) -- (P01);
\draw [SoThick] (P01) -- (P00);

\draw (P10) -- (P20);
\draw [SoThick] (P20) -- (P21);
\draw (P21) -- (P11);

\draw (P21) -- (P22);
\draw [SoThick] (P22) -- (P12);
\draw (P12) -- (P11);
\end{tikzpicture}
\caption{The associated snake diagram of the quiver $1\to2\to3$ with a perfect matching.}
\end{figure}
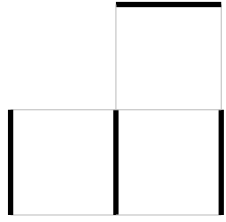

For the fixed completely extended linear quiver $(Q,Q')$, let $\M$ be the set of all perfect matchings in the associated snake diagram, and $\G$ be the set of all GCCs. We shall construct a bijective map $\psi_{\M,\G}: \M\to \G$ and its inverse $\psi_{\G,\M}$. First we prove a simple lemma.
\begin{lemma} \label{InTheSamePerfectMatching}
For any perfect matching $\gamma$ and $i\in [0,n]$, there is exactly one edge of $\gamma$ lies in $Pl^{(i)}$.
\end{lemma}
\begin{proof}
The statement is obviously true for $i=0$ and $i=n$.
Suppose the statement is false for some $i\in [1,n-1]$. Since $\gamma$ is a perfect matching, we are in one of the following two cases.

Case 1: both $T_i^{(i)}, T_{i+1}^{(i)}$ are in $\gamma$. If we remove $Pl^{(i)}$ (4 vertices and 3 edges), then the rest graph has two components which have odd number of vertices and have perfect matchings. This is a contradiction.

Case 2: none of  the three edges $T_i^{(i)}, T_{i+1}^{(i)}, T_{i,i+1}$ lies in $\gamma$. Then we remove the three edges (but do not remove the vertices) and apply the same argument as in Case 1.
\end{proof}

\begin{remark}
Thanks to the above lemma, we can write a perfect matching $\gamma$ as $\{\gamma_1,\dots,\gamma_n\}$ where $\gamma_i\in Pl^{(i)}$.
\end{remark}

\begin{definition}-
(i) We define a map $\psi_{\M,\G}:\M\to\G$ by sending $\gamma\in \M$ to $\{S_{i,r}\}\in\G$ such that
for $i\in[1,\lqv-1]$,
$$(|S_{i,1}|, |S_{i,2}|)
=\begin{cases}
(\delta_i,1-\delta_i), & \text{ if } T_i^{(i)}\in\gamma,\\
(1-\delta_i,\delta_i), & \text{ if } T_{i+1}^{(i)}\in\gamma,\\
(0,0), & \text{ if } T_{i,i+1}\in\gamma.\\
\end{cases}$$
(By Lemma \ref{InTheSamePerfectMatching}, exactly one of the three cases occurs.)

(ii) We define a map $\psi_{\G,\M}:\G\to\M$ by sending $\{S_{i,r}\}\in\G$ to the set of edges $\gamma=\{\gamma_0,\gamma_1,\ldots,\gamma_n\}$ such that
$$\gamma_i=\left\{\begin{array}{lll}
T^{(i)}_{i}, & \text{if} & (|S_{i,1}|, |S_{i,2}|)=(\delta_i,1-\delta_i) \xh
T^{(i)}_{i+1}, & \text{if} & (|S_{i,1}|, |S_{i,2}|)=(1-\delta_i,\delta_i) \xh
T_{i,i+1}, & \text{if} & (|S_{i,1}|, |S_{i,2}|)=(0,0) \xh
\end{array}\right.$$
for $i\in[1,\lqv-1]$, and

$\gamma_0=\left\{\begin{array}{lll}
T_{1,0} & \text{if  } |S_{1,1+\delta_1}| = 1-\delta_1,\\
T_{1,1} & \text{otherwise},
\end{array}\right.$ \quad
$\gamma_n=\left\{\begin{array}{lll}
T_{\lqv,0} & \text{ if }  |S_{n-1,2-\delta_{n-1}}|=\delta_{n-1},\\
T_{\lqv,1} & \text{ otherwise}.
\end{array}\right.$

\end{definition}

We assign a weight $w(u)$ for each edge $u$ of the snake diagram as follows for all $i,j$:
\begin{equation}\label{weight w}
w(T^{(i)}_j)=x_j,\quad w(T_{j,j+1})=x_{j,j+1}
\end{equation}
For a perfect matching $\gamma=(\gamma_0,\dots,\gamma_n)$, define its weight $w(\gamma)=\prod_{i=0}^n w(\gamma_i)$. In \cite{MS} it is proved that the cluster variable with \textbf{d}-vector ${\bf a}$ is
\begin{equation}\label{MS}
x[\bfa]=\left(\prod_{i=1}^n x_{i}^{-1}\right) \sum_\gamma w(\gamma).
\end{equation}

\begin{theorem}\label{PMsGCCs}
The maps $\psi_{\M,\G}$ and $\psi_{\G,\M}$ are well-defined and are inverses of each other. Moreover, $w(\gamma_i)=y_i$, thus $\psi_{\M,\G}$ induces a bijective proof of Theorem \ref{MainThm} using \eqref{MS}.
\end{theorem}
\begin{proof}


(i) We show that $\psi_{\M,\G}$ is well-defined, that is, $\psi_{\M,\G}(\gamma)=\{S_{i,r}\}$ satisfies the condition \eqref{GCCRule}.  It's clear from the construction that for every $i\in [1,n-1]$, $(|S_{i,1}|,|S_{i,2}|)\ne(1,1)$. Next, we prove (a) and (c) of \eqref{GCCRule}, since (b) and (d) can be proved similarly.

For (a), we suppose $(\delta_{i-1},\delta_{i})=(0,0)$ and need to show that $T^{(i-1)}_{i-1}\in\gamma\Leftrightarrow T^{(i)}_{i+1}\notin\gamma$. This is true because the two edges $T^{(i-1)}_{i-1}$ and $T^{(i)}_{i+1}$  are incident to the same deg-2 vertex, thus exactly one of them is in $\gamma$. (See the left diagram in Figure \ref{labels}.)

For (c), we suppose $(\delta_{i-1},\delta_i)=(0,1)$ and need to show that $T^{(i-1)}_{i-1}\in\gamma \Leftrightarrow T^{(i)}_{i+1}\in\gamma$. These two edges are opposite edges of a tile which is the middle of three tiles in a row or a column. Deleting these two edges will separate the snake diagram into two graphs with even number of vertices each. Thus the two edges must be both in $\gamma$ or not in $\gamma$. (See the right diagram in Figure \ref{labels}.)

\noindent (ii) We show that $\psi_{\G,\M}$ is well-defined, that is, $\psi_{\G,\M}(\{S_{i,r}\})=\gamma$ is a perfect matching. Since the snake diagram has $2n+2$ vertices and $\gamma$ has $n+1$ edges, it suffices to show that all edges in $\gamma$ are disjoint. We assume the contrary that $\gamma_c$ shares a vertex with $\gamma_d$ for some $0\le c<d\le n$. Since $\gamma_c\in Pl^{(c)}$ and  $\gamma_d\in Pl^{(d)}$, $Pl^{(c)}$ and $Pl^{(d)}$ much be consecutive, thus $d=c+1$.

We first assume $1\le c\le n-2$. We shall discuss two cases $(\delta_c,\delta_{c+1})=(0,0)$ and $(0,1)$, and omit the other two cases $(1,0)$ and $(1,1)$ since the proof is similar.

\noindent Case $(\delta_c,\delta_{c+1})=(0,0)$: since $\{S_{i,r}\}$ is a GCC, we must have
$(|S_{c,1}|,|S_{c,2}|,|S_{c+1,1}|,|S_{c+1,2}|)=(0,1,0,1)$, $(0,1,0,0)$, $(1,0,1,0$, or $(0,0,1,0)$. Correspondingly,
$$(\gamma_c, \gamma_{c+1})=(T_c^{(c)}, T_{c+1}^{(c+1)}), (T_{c}^{(c)}, T_{c+1,c+2}), (T_{c+1}^{(c)}, T_{c+2}^{(c+1)}), \textrm{ or } (T_{c,c+1}, T_{c+2}^{(c+1)}).$$ It is obvious from Figure \ref{cc+1c+2} that $\gamma_c$ and $\gamma_{c+1}$ are disjoint, a contradiction as expected.

\begin{figure}[h!]
\begin{tikzpicture}[MyPic,scale=1.5]
\tikzstyle{every node}=[font=\tiny]
\tsomepoints
	
\draw (P00) -- (P10);
\draw (P11) -- (P10) node[midway,above,sloped] {$T_{c,c+1}$};
\draw (P11) -- (P01) node[midway,above] {$T_{c+1}^{(c)}$};
\draw (P01) -- (P00);
\node at (0.5,0.5) [fill=lightgray] {$c$};
	
\draw (P10) -- (P20) node[midway,below] {$T_c^{(c)}$};
\draw (P20) -- (P21) node[midway,right] {$T_{c+2}^{(c+1)}$};
\draw (P21) -- (P11) node[midway,above] {$T_{c+1,c+2}$};
\node at (1.5,0.5) [fill=lightgray] {$c+1$};
	
\draw (P21) -- (P22);
\draw (P22) -- (P12);
\draw (P12) -- (P11) node[near start,left] {$T_{c+1}^{(c+1)}$};
\node at (1.5,1.5) [fill=lightgray] {$c+2$};


\begin{scope}[shift={(4,0)}]
\tikzstyle{every node}=[font=\tiny]
\tsomepoints

\draw (P00) -- (P10);
\draw (P11) -- (P10) node[midway,above,sloped] {$T_{c,c+1}$};
\draw (P11) -- (P01) node[midway,above] {$T_{c+1}^{(c)}$};
\draw (P01) -- (P00);
\node at (0.5,0.5) [fill=lightgray] {$c$};

\draw (P10) -- (P20) node[midway,below] {$T_c^{(c)}$};
\draw (P21) -- (P20) node[midway,above,sloped] {$T_{c+1,c+2}$};
\draw (P21) -- (P11) node[midway,above] {$T_{c+2}^{(c+1)}$};
\node at (1.5,0.5) [fill=lightgray] {$c+1$};

\draw (P20) -- (P30) node[midway,below] {$T_{c+1}^{(c+1)}$};
\draw (P30) -- (P31);
\draw (P31) -- (P21);
\node at (2.5,0.5) [fill=lightgray] {$c+2$};

\end{scope}
\end{tikzpicture}
\caption{Left: $(\delta_c,\delta_{c+1})=(0,0)$ and Right: $(\delta_c,\delta_{c+1})=(0,1)$}
\label{cc+1c+2}
\end{figure}
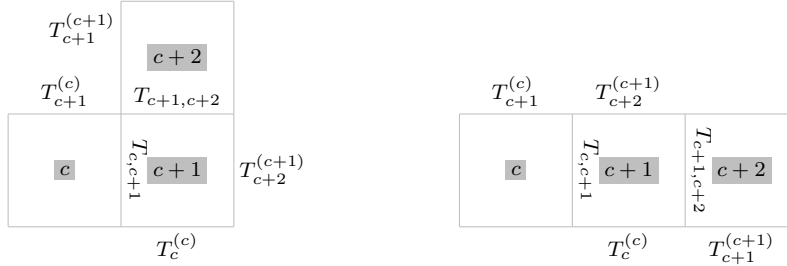

\noindent Case $(\delta_c,\delta_{c+1})=(0,1)$: similar as the above case,
$$(\gamma_c, \gamma_{c+1}) = (T_{c}^{(c)}, T_{c+2}^{(c+1)}), (T_{c+1}^{(c)}, T_{c+1}^{(c+1)}), (T_{c+1}^{(c)}, T_{c+1,c+2}), (T_{c,c+1}, T_{c+1}^{(c+1)}) \text{ or } (T_{c,c+1}, T_{c+1,c+2}).$$
We get the expected contradiction by observing Figure \ref{cc+1c+2}.

The cases of $c=0$ and $c=n-1$ are proved by a similar argument.

\noindent(iii) The fact that $\psi_{\M,\G}$ and $\psi_{\G,\M}$ are inverse of each other, and $w(\gamma_i)=y_i$, follows easily from their definitions.
\end{proof}

\begin{example}
If $\gamma=\{T_1,T_3,T_6,T_8,T_{10}\}$ then $\psi_{\M,\G}(\gamma)=((0,1),(0,0),(1,0))$ as shown in Figure \ref{exMapMG}.

\begin{figure}[h!]
\begin{tikzpicture}[MyPic,scale=1.3]
	\tikzstyle{every node}=[font=\tiny]
	\tsomepoints
	
	\draw (P00) -- (P10) node[midway,below] {$T_5$};
	\draw (P10) -- (P11) node[midway,right] {$T_7$};
	\draw (P11) -- (P01) node[midway,above] {$T_2^{(1)}$};
	\draw [SoThick] (P01) -- (P00) node[midway,left] {$T_6$};
	
	\draw [SoThick] (P10) -- (P20) node[midway,below] {$T_1^{(1)}$};
	\draw (P20) -- (P21) node[midway,right] {$T_3^{(2)}$};
	\draw [SoThick] (P21) -- (P11) node[midway,above] {$T_8$};
	
	\draw (P21) -- (P22) node[midway,right] {$T_4^{(3)}$};
	\draw (P22) -- (P12) node[midway,above] {$T_9$};
	\draw (P12) -- (P11) node[midway,left,near start] {$T_2^{(2)}$};
	
	\draw [SoThick] (P22) -- (P23) node[midway,right] {$T_{10}$};
	\draw (P23) -- (P13) node[midway,above] {$T_{11}$};
	\draw [SoThick] (P13) -- (P12) node[midway,left] {$T_3^{(3)}$};
	
\begin{scope}[shift={(3,1.5)},scale=1]
\node (v1) at (0:0) {};
\node (v2) at (0:2) {};

\draw [|->,red] (v1) -- (v2) node[midway,above,font=\normalsize] {$\psi_{\M,\G}$};
\end{scope}

\begin{scope}[shift={(4,1.2)},scale=0.8]
\tsomepoints
\tikzstyle{every node}=[font=\normalsize]

\draw (P20) -- (P30) node[midway,below] {$x_2$};
\draw [SoThick] (P30) -- (P31) node[midway,right] {$x_1$};

\draw (P40) -- (P50) node[midway,below] {$x_3$};
\draw (P50) -- (P51) node[midway,right] {$x_2$};

\draw (P60) -- (P61) node[midway,right] {$x_4$};
\draw [SoThick] (P60) -- (P70) node[midway,below] {$x_3$};

\end{scope}
\end{tikzpicture}
\caption{An example of the map $\psi_{\M,\G}$}
\label{exMapMG}
\end{figure}
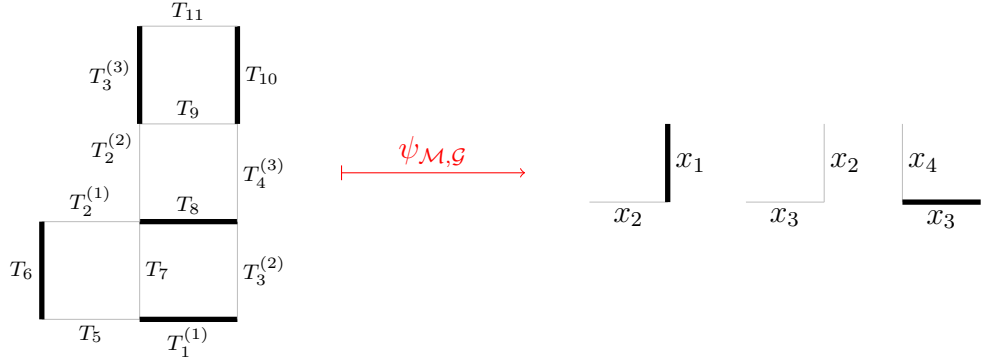
\end{example}

\section{Proof of Main Theorems \ref{thm:01sequence} and \ref{thm:GCCxa}}
We have explained in \S\ref{subsection3.3} that,  in order to compute cluster variables, it suffices to have the formula for a completely extended linear quiver, namely Theorem \ref{MainThm}. This theorem follows from Theorem \ref{PMsGCCs}.
 In this section, we show how to derive Theorem \ref{thm:01sequence} and Theorem \ref{thm:GCCxa} from Theorem \ref{MainThm}.

\subsection{Proof of Theorem \ref{thm:GCCxa}}
Let $x'[{\bf a}]$ be the right hand side of the formula in Theorem \ref{thm:GCCxa}.
We shall show that (i) the GCCs for the \textbf{d}-vector ${\bf a}$ are in one-to-one correspondence with the collections of GCCs for the \textbf{d}-vectors ${\bf b}$'s described in \eqref{b}; (ii)  $x'[{\bf a}]=\prod_{\bf b} x'[{\bf b}]$; and (iii) $x'[{\bf b}]=x[{\bf b}]$ using Theorem \ref{MainThm}.
It then follows that $x'[{\bf a}]=x[{\bf a}]$.

(i) Let $\{S^{(i,j)}_\ell\}$ be any GCC for the \textbf{d}-vector ${\bf a}$. For each pipeline $\Lambda$, let ${\bf b}={\bf b}_\Lambda$, we construct a GCC $\{S^{(i,j),\Lambda}_\ell\}$ for the \textbf{d}-vector ${\bf b}$ by requiring the following for each arrow $i\to j$:

-- if $\Lambda$ intersects the edge $i$ at the $r$-th marked point, then $|S^{(i,j),\Lambda}_1|=1$ if and only if $u^{(i,j)}_r\in S^{(i,j)}_1$,

-- if $\Lambda$ intersects the edge $j$ at the $r$-th marked point, then $|S^{(i,j),\Lambda}_2|=1$ if and only if $v^{(i,j)}_r\in S^{(i,j)}_2$,

\noindent (in both case the marked points are ordered in the increasing distance from the common endpoint of $i$ and $j$).

To verify that $\{S^{(i,j),\Lambda}_\ell\}$ is a GCC for the \textbf{d}-vector ${\bf b}$, we need to check the conditions in Definition \ref{definition:GCC}. The only nontrivial condition to check is that for a 3-cycle $k\to i\to j\to k$, $S^{(i,j),\Lambda}_1$ and $S^{(i,j),\Lambda}_2$ are $\sigma_{ijk}^{\bf b}$-compatible. That reduces to showing that $(|S^{(i,j),\Lambda}_1|,|S^{(i,j),\Lambda}_1|)\neq(1,1)$ in the case $(b_i,b_j,b_k)=(1,1,0)$. In this case, $\Lambda$ intersects edges $i$ and $j$ at the $r$-th marked points for some $r\le \sigma_{ijk}^{\bf a}$. Then either $u^{(i,j)}_r\notin S^{(i,j)}_1$ or $v^{(i,j)}_r\notin S^{(i,j)}_2$. In the former case, $|S^{(i,j),\Lambda}_1|=0$; in the latter case, $|S^{(i,j),\Lambda}_2|=0$. Therefore $(|S^{(i,j),\Lambda}_1|,|S^{(i,j),\Lambda}_2|)\neq(1,1)$.

It is easy to see that a unique GCC $\{S^{(i,j)}_\ell\}$ is determined if we take any collection of GCCs $\{S^{(i,j),\Lambda}_\ell \}$ for all pipelines $\Lambda$. So we have the desired one-to-one correspondence.

(ii)  We show that $\prod_{\bf b} x'[{\bf b}]=x'[{\bf a}]$. Since $\sum {\bf b}={\bf a}$, it suffices to show that, for each GCC $\{S^{(i,j)}_\ell\}$, letting $\{S^{(i,j),\Lambda}_\ell\}$ be defined as in (i), the following holds (recall that ${\bf b}={\bf b}_\Lambda$ depends on $\Lambda$):
$$
\prod_\Lambda\left(\prod_{i\to j}  x_{i}^{\big|S^{(i,j),\Lambda}_2\big|} x_{j}^{\big|S^{(i,j),\Lambda}_1\big|} \cdot \prod_{i\to j\to k\to i} x_i^{-\sigma^{\bf b}_{jki}}\right)
=
\prod_{i\to j}  x_{i}^{\big|S^{(i,j)}_2\big|} x_{j}^{\big|S^{(i,j)}_1\big|}\cdot \prod_{i\to j\to k\to i} x_i^{-\sigma^{\bf a}_{jki}}.
$$
Since the left hand side is equal to
$$
\prod_{i\to j}  x_{i}^{\sum_\Lambda\big|S^{(i,j),\Lambda}_2\big|} x_{j}^{\sum_\Lambda\big|S^{(i,j),\Lambda}_1\big|} \cdot \prod_{i\to j\to k\to i} x_i^{-\sum_\Lambda\sigma^{\bf b}_{jki}},$$
it suffices to show that $\sum_\Lambda |S^{(i,j),\Lambda}_2| = |S^{(i,j)}_2|$,  $\sum_\Lambda |S^{(i,j),\Lambda}_1| = |S^{(i,j)}_1|$, and $\sum_\Lambda \sigma^{\bf b}_{jki}=\sigma^{\bf a}_{jki}$. The first two are clear. To show the last equality:  first note that if $\Lambda$ is disjoint from the edge $j$, then $b_j=0$ and thus $\sigma^{\bf b}_{jki}=0$. So we only need to consider those $\Lambda$'s that intersect $j$. Let $\Lambda_r$ ($1\le r\le a_j$) be the pipeline that intersects $j$ at the $r$-th marked point (ordered in the increasing distance to the common endpoint of $j$ and $k$). If $r\le \sigma^{\bf a}_{jki}$, then $(b_j,b_k,b_i)=(1,1,0)$, thus $\sigma^{\bf b}_{ijk}=1$; otherwise, either $b_j=0$ or $b_k=0$, thus  $\sigma^{\bf b}_{ijk}=0$. Therefore $\sum_\Lambda \sigma^{\bf b}_{jki}=\sigma^{\bf a}_{jki}$.

(iii) We show that $x'[{\bf b}]=x[{\bf b}]$. By Remark \ref{rmk2.1}, it suffices to show that, in the setting of Theorem \ref{MainThm}, the right hand side of \eqref{formula3} is equal to $x'[{\bf a}]$. It breaks down to show that, for $i'\in[0,n]$, the following equality holds for the $i'$-th 3-cycle $i\to j\to k\to i$ in $Q'$
(for $i'\in[1,n-1]$, the $i'$-th 3-cycle is the one that contains vertices $v_i$, $v_{i+1}$ and $v_{i,i+1}$;  the $0$-th 3-cycle is $v_1\to v_{1,0}\to v_{1,1}\to v_1$; the $n$-th 3-cycle is $v_n\to v_{n,0}\to v_{n,1}\to v_n$):
\begin{equation}\label{eq:proof6.2}
\Big(x_{i}^{\big|S^{(i,j)}_2\big|} x_{j}^{\big|S^{(i,j)}_1\big|}\Big)
\Big(x_{j}^{\big|S^{(j,k)}_2\big|} x_{k}^{\big|S^{(j,k)}_1\big|}\Big)
\Big(x_{k}^{\big|S^{(k,i)}_2\big|} x_{i}^{\big|S^{(k,i)}_1\big|}\Big)
\cdot
x_i^{-\sigma_{jki}}
x_j^{-\sigma_{kij}}
x_k^{-\sigma_{ijk}}
=y_{i'}
\end{equation}
We shall only prove the case when $i'\in [1,n-1]$ and $\delta_{i'}=0$, because other cases can be proved in a similar way. In this case, the $i'$-th 3-cycle is $v_i\to v_{i+1}\to v_{i,i+1}\to v_i$ (where $i=i'$), and the left hand side of \eqref{eq:proof6.2} is equal to
$$\Big(x_{i}^{\big|S_{i,2}\big|} x_{j}^{\big|S_{i,1}\big|}\Big)
\Big(x_{j}^{0} x_{k}^{1-\big|S_{i,2}\big|}\Big)
\Big(x_{k}^{1-\big|S_{i,1}\big|} x_{i}^{0}\Big)
\cdot
x_k^{-1}
=
x_{i}^{\big|S_{i,2}\big|} x_{j}^{\big|S_{i,1}\big|}
x_{k}^{1-\big|S_{i,1}\big|-\big|S_{i,2}\big|}=y_{i'}.
$$
So \eqref{eq:proof6.2} holds.

\subsection{Proof that Theorem \ref{thm:GCCxa} implies Theorem \ref{thm:01sequence}}
We give a bijection between GCSs and GCCs. Let ${\bf s}$ be a GCS. Consider a 3-cycle $i\to j\to k\to i$, labeled in the way that $d(k)<d(i)<d(j)$. Then we define
$$\aligned
&S^{(k,i)}_1=\{u_r\in \mathcal{D}^{(k,i)}_1 \big| s_{k,r}=1\},\quad\quad
S^{(k,i)}_2=\{v_r\in \mathcal{D}^{(k,i)}_2 \big| s_{i,r}=0\}\\
&S^{(i,j)}_1=\{u_r\in \mathcal{D}^{(i,j)}_1 \big| s_{i,a_i+1-r}=1\},\quad
S^{(i,j)}_2=\{v_r\in \mathcal{D}^{(i,j)}_2 \big| s_{j,r}=0\}\\
&S^{(j,k)}_1=\{u_r\in \mathcal{D}^{(j,k)}_1 \big| s_{j,a_j+1-r}=1\},\quad
S^{(j,k)}_2=\{v_r\in \mathcal{D}^{(j,k)}_2 \big| s_{k,a_k+1-r}=0\}\\
\endaligned
$$
It is then easy to check that the conditions of GCSs and GCCs, as well as the two theorems, are equivalent under this bijection.

\section{A bijection between GCSs and broken lines}\label{S7}
This section is devoted to the construction of a bijection between GCSs and broken lines in the even rank case (Theorem \ref{MainThm2}) and the general case (Theorem \ref{MainThm2'}). 
In \S\ref{subsection:g-vectors}, we give a description of the g-vector of a cluster variable, which will determine the initial direction of the broken lines. In \S\ref{subsection:Scattering diagrams and broken lines}, we recall the necessary facts on scattering diagrams, broken lines, and theta functions. The rest of the section is to state and prove Theorem \ref{MainThm2} and \ref{MainThm2'}.
 
In this section,  we let $\tilde{Q}$ be a type $A$ quiver and $Q$ be a linear full subquiver of $\tilde{Q}$. 
Relabelling vertices if necessary, we assume that $Q$ is 
\begin{equation}\label{labelQ1ton}
\textrm{$Q$: \quad $1\longleftrightarrow 2 \longleftrightarrow\cdots \longleftrightarrow n$\quad (each arrow can go either direction).}
\end{equation}

\subsection{g-vectors} \label{subsection:g-vectors}

It is shown in \cite[Theorem 7.5 (4)]{GHKK} that, if $m$ is the g-vector of a cluster variable, then the cluster variable is exactly $\vartheta_m$. So we first study g-vectors.

\begin{lemma}\label{lem:g}
Let ${\bf a}_Q$ be defined as in \eqref{df:a}. Then the g-vector $m_0=(g_r)$ of the cluster variable $x[{\bf a}_Q]$ is given by
$$ g_r= \left\{\begin{array}{lr}
        \deg_Q^-(r)-1, \quad \text{if $r\in Q_0$;}\\
        1, \quad \text{if $r\notin Q_0$  and $(\deg_Q^+(r),\deg_Q^-(r))=(0,1)$;}\\
        0, \quad \text{if $r\notin Q_0$ and  $(\deg_Q^+(r),\deg_Q^-(r))\neq(0,1)$.}\\
        \end{array}\right.
$$        
\end{lemma}	

\begin{proof}
We use the description of g-vectors in \cite[\S13.1]{MSW}:
$$g_\gamma=\deg\frac{x(P_-)}{{\rm cross }(T^\circ,\gamma)}$$
It is easy to verify that $P_-$ corresponds to the GCS ${\bf s}=(s_i)$ where $s_i=1$ if and only if $i\in Q$ (this is the unique GCS satisfying $\bar{s}_i=0$). 
Thus
\begin{equation}\label{xP-}
x(P_-)=\frac{\left(\prod_{i\to j}  x_i^{\bar{s}_j}x_j^{s_i}\right) }{ \prod_{k\in K}x_k}
=\frac{ \prod_{i\to j}  x_j^{s_i} }{ \prod_{k\in K}x_k} 
\end{equation}
Meanwhile, ${\rm cross }(T^\circ,\gamma)=\prod_{l=1}^n  x_l^{-a_l}$.
So we can compute the g-vector case by case:

If $r\in Q$: the power of $x_r$ in \eqref{xP-} is equal to the number of arrows $i\to r$ with $s_r=1$, which is $\deg^-_Q(r)$. Since $a_r=1$, $g_r=\deg^-_Q(r)$.

If $r\notin Q$  and $(\deg_Q^+(r),\deg_Q^-(r))=(0,1)$: then $r\notin K$, thus the power of $x_r$ in \eqref{xP-}  is $\deg^-_Q(r)=1$.

If $r\notin Q$  and $(\deg_Q^+(r),\deg_Q^-(r))\neq (0,1)$: then $(\deg_Q^+(r),\deg_Q^-(r))=(1,1), (1,0)$ or $(0,0)$.  If it is $(1,1)$, then $r\in K$, thus the power of $x_r$ in \eqref{xP-}  is $\deg^-_Q(r)-1=0$; if it is $(1,0)$ or $(0,0)$, then  $\deg_Q^-=0$ and $r\notin K$, thus $x_r$ does not appear in $\eqref{xP-}$. In all cases, we get $g_r=0$.
\end{proof}

\begin{remark}
The above is equivalent  to the following description of the g-vector of $x[{\bf a}_Q]$: 
 $$ g_r= \left\{\begin{array}{lr}
        -1, \;\text{if $r\in Q$ is not the head of any arrow in $Q$;}\\
        1, \quad \text{if ``$r\in Q$ is the head of two arrows in $Q$'', or ``$r\notin Q$ is adjacent to only }\\
        \quad\quad \text{one vertex in $Q$, either $1\to r$ or $n\to r$'';}\\
        0, \quad \text{otherwise.}\\
        \end{array}\right.
$$  
\end{remark}

\subsection{Scattering diagrams and broken lines}\label{subsection:Scattering diagrams and broken lines}
We only recall the necessary facts needed in our paper, specialized in our setting. For more reference, see \cite{GHKK}. 

Recall that in \S\ref{subsection:quiver def} we defined $I=I_{\rm uf}=\{1,\dots,n\}$ for a coefficient free cluster algebra $\mathcal{A}$ of rank $n$, and $I=\{1,\dots,2n\}$ and $I_{\rm uf}=\{1,\dots,n\}$ for a principal coefficient cluster algebra $\mathcal{A}_{\rm prin}$ of rank $n$. 

Let $M\cong\mathbb{Z}^{n'}$, $N={\rm Hom}(M,\mathbb{Z})$ with a basis $(e_i)_{i\in I}$, $M_\mathbb{R}:=M\otimes \mathbb{R}$ with dual basis $(f_i)_{i\in I}$. $N$ is equipped with a skew-symmetric form $\{\cdot,\cdot\}$. Let $N_{\rm uf}$ be a sublattice of $N$ with basis $(e_i)_{i\in I_{\rm uf}}$. Define
$$N^+:=\big\{\sum_{i\in I_{\rm uf}} a_ie_i |  a_i\ge 0, \sum a_i>0\big\}.$$

We assume the Fundamental Assumption: the map $p_1^*: N_{\rm uf}\to M$ given by $n\mapsto \{n,\cdot\}$ is injective. 

Choose a strictly convex top-dimensional cone $\sigma\subseteq M_\mathbb{R}$, with associated monoid $P:=\sigma\cap M$, such that $p_1^*(e_i)\in P\setminus\{0\}$ for all $i\in I_{\rm uf}$. Let $\widehat{\mathbb{Z}[P]}$ be the completion of $\mathbb{Z}[P]$ at the maximal monomial ideal $\mathfrak{m}$ generated by $\{x^m|m\in P\setminus\{0\}\}$. 

A \emph{wall} is a pair $(\mathfrak{d}, f_\mathfrak{d})$ where $\sigma\in M_\mathbb{R}$ is a $({\rm rank } N-1)$-dimensional convex rational polyhedral cone contained in $n_0^\perp$ for some primitive vector   $n_0\in N^+$, and  $f_\mathfrak{d}=1+\sum_{k>0} c_kz^{kp_1^*(n_0)}$. 

A \emph{scattering diagram} $\mathfrak{D}$ is a collection of walls such that there are only finite many walls $(\mathfrak{d}, f_\mathfrak{d})\in \mathfrak{D}$ satisfying $f_\mathfrak{d}\not\equiv 1\mod \mathfrak{m}^k$ for each $k>0$. 

Given a scattering diagram $\mathfrak{D}$, define its \emph{support} to be
$${\rm Supp}(\mathfrak{D}):=\bigcup \mathfrak{d},$$
and define its \emph{singular locus} to be
$${\rm Sing}(\mathfrak{D}):=\bigcup \partial\mathfrak{d}\cup \bigcup_{\dim \mathfrak{d}_1\cap\mathfrak{d}_2=n-2} \mathfrak{d}_1\cap \mathfrak{d}_2.$$

For a smooth path $\gamma: [0,1]\to M_{\mathbb{R}}\setminus{\rm Sing}(\mathfrak{D})$ whose endpoints are not in ${\rm Supp}(\mathfrak{D})$ and that it is transversal to each wall that it crosses, we define an automorphism $\theta_{\gamma,\mathfrak{D}}$ of $\widehat{\mathbb{Z}[P]}$ as follows:
for each $k>0$, we can find numbers $0<t_1\le t_2\le\cdots\le t_s<1$ such that $\gamma(t_i)\in\mathfrak{d}_i$, with $(\mathfrak{d}_i, f_{\mathfrak{d}_i})\in \mathfrak{D}$, $f_{\mathfrak{d}_i}\not\equiv 1\mod \mathfrak{m}^k$, and $\mathfrak{d}_i\neq\mathfrak{d}_j$ if $t_i=t_j$, and $s$ taken as large as possible. For each $i$, define an automorphism $\theta_i$ to be
$$\theta_i(x^m):=x^mf_{\mathfrak{d}_i}^{m\cdot n_i}$$
where $n_i$ is the primitive vector annihilating $\mathfrak{d}_i$ and satisfying $\gamma'(t_i)\cdot n_i<0$. Define
$$\theta^{(k)}_{\gamma,\mathfrak{D}}=\theta_s\theta_{s-1}\cdots\theta_1,\quad\textrm{ and }\;
\theta_{\gamma,\mathfrak{D}}=\lim_{k\to\infty}\theta^{(k)}_{\gamma,\mathfrak{D}}.$$

A \emph{consistent scattering diagram} $\mathfrak{D}$ is a scattering diagram such that $\theta_{\gamma,\mathfrak{D}}$ only depends on the endpoints of $\gamma$. 

Two scattering diagrams $\mathfrak{D}$, $\mathfrak{D}'$ are \emph{equivalent} if $\theta_{\gamma,\mathfrak{D}}=\theta_{\gamma,\mathfrak{D}'}$ for all paths $\gamma$ for which both sides are defined. Let $p^*: N\to M$ be given by $n\mapsto \{n,\cdot\}$. A wall $\mathfrak{d}$ is \emph{incoming} if $p^*(n_0)\in \mathfrak{d}$, where 
$n_0\perp \mathfrak{d}$ and $n_0\in N^+$; otherwise, the wall is \emph{outgoing}. 
Define the initial scattering diagram 
\begin{equation}\label{eq:initial walls}
\mathfrak{D}_{\rm in}:= \{(e_i^\perp,1+z^{v_i})| i\in I_{\rm uf}\},\quad \textrm{ where } v_i:=p_1^*(e_i).
\end{equation}
Note that all walls in $\mathfrak{D}_{\rm in}$ are incoming.  

\begin{definition}\label{def of D}
Let $\mathfrak{D}\supset\mathfrak{D}_{\rm in}$ be a consistent scattering diagram such that $\mathfrak{D}\setminus\mathfrak{D}_{\rm in}$ consists only of outgoing walls.  
\end{definition}
(It is proved in \cite[Theorem 1.7]{GHKK} that $\mathfrak{D}$ exists and is unique up to equivalence.)

\smallskip
Let $\mathfrak{D}$ be as defined in Definition \ref{def of D}, $m_0\in M\setminus\{0\}$ and $\QQ\in M_\mathbb{R}\setminus{\rm Supp}(\mathfrak{D})$. A \emph{broken line} for $m_0$ with endpoint $\QQ$ is a piecewise linear continuous proper path $\gamma:(-\infty,0] \to \mathbb{R}^n\setminus{\rm Sing }(\mathfrak{D})$ with a finite number of domains of linearity. This path comes along with the data of, for each domain of linearity $L\subseteq (-\infty,0]$ of $\gamma$, a monomial $c_Lz^{m_L}\in \mathbb{Z}[M]$. This data satisfies the following properties:
\begin{enumerate}[itemsep=5pt]
\item $\gamma(0)=\QQ$.
\item If $L$ is the first (and therefore unbounded) domain of linearity	of $\gamma$, then $c_Lz^{m_L}=z^{m_0}$.
\item For $t$ in a domain of linearity $L$, $\gamma'(t)=-m_L$.
\item Let $t\in (-\infty,0)$ be a point at which $\gamma$ is not linear, passing from domain of linearity $L$ to $L'$.
Let $\mathfrak{D}_t=\{(\mathfrak{d},f_{\mathfrak{d}})\in\mathfrak{D} | \gamma(t)\in\mathfrak{d}\}$. Then $c_{L'}z^{m_{L'}}$ is a term in the formal power series 
$$c_Lz^{m_L}\prod_{(\mathfrak{d},f_{\mathfrak{d}})\in\mathfrak{D}_t} f_\mathfrak{d}^{|\langle n_0,m_L\rangle|}.$$ 
\end{enumerate}
We denote the monomial attached to the final domain of linearity of $\gamma$ by ${\rm Mono }(\gamma)$.

For given $m_0\in M\setminus\{0\}$ and $\QQ\in M_\mathbb{R}\setminus{\rm Supp}(\mathfrak{D})$,  the \emph{theta function} is defined to be
$$\vartheta_{\QQ,m_0}=\sum_\gamma {\rm Mono}(\gamma),$$
where $\gamma$ runs over all broken lines for $m_0$ with endpoint $\QQ$.

\begin{remark} \label{conditions not needed}
The broken lines that we shall construct in this paper will not bend on walls that are not in the initial scattering diagram $\mathfrak{D}_{\rm in}$ defined in \eqref{eq:initial walls}. A priori, there may exist broken lines that do bend on walls in $\mathfrak{D}\setminus\mathfrak{D}_{\rm in}$ (i.e., outgoing walls); but we shall argue that there are no such broken lines in our setting (where we only consider those appearing in the theta function corresponding to cluster variables; see \S\ref{subsection:bijective}). In general, broken lines can also bend on outgoing walls if we consider those appearing in the theta function corresponding to cluster monomials. Because of this obstacle, we could not yet extend the cluster variable formulas Theorem \ref{MainThm2} and \ref{MainThm2'} (which explicitly decribe the broken lines using GCSs) to cluster monomials.
\end{remark}

\subsection{A bijection between GCSs and broken lines (in the even rank case)}

In this subsection, we give a bijection between GCSs and broken lines for the cluster variable $x[{\bf a}_Q]$ when $n'$ (the rank of the cluster algebra $\mathcal{A}$) is even. This is exactly the case when the exchange matrix $B$ is of full rank, which guarantees that the Fundamental Assumption in \cite{GHKK} is satisfied. (Indeed, since \cite[Lemma 3.2]{BFZ} asserts that the rank is invariant under mutation, it is suffices to consider the linear quiver $1\to 2\to\cdots\to n'$, in which case the determinant of $B$ is 1 if $n'$ is even, 0 if $n'$ is odd.) 

\begin{definition}\label{df:adjustableposition}
For a given GCS ${\bf s}$, we say that $r\in [1,n]$ is an {\it adjustable} position if 

(a) the $r$-th coordinate of ${\bf s}$ is $0$, and 

(b) the sequence obtained from ${\bf s}$ by replacing the $r$-th coordinate by $1$ is still a GCS. 

\end{definition}

\begin{remark}\label{rmk:adjustableposition}
An adjustable position always exists as long as some coordinate of ${\bf s}=(s_j)$ is 0. This follows from the following equivalent definition:  let $Q^{\bf s}$ be the full subquiver of $Q$ with vertex set $\{j \in Q_0\,|\, s_j=0\}$. Then $r$ is an adjustable position if and only if $r$ is a sink of $Q^{\bf s}$.
Since $Q$ is acyclic, $Q^{\bf s}$ is also acyclic, thus $Q^{\bf s}$ has at least one sink.
\end{remark}

\begin{definition}\label{df:wisi}
For a given GCS ${\bf s}$, let $\ell=n-|{\bf s}|$. Define ${\bf s}^{(\ell)}={\bf s}$, and define $w_i$ ($1\le i\le \ell$) and ${\bf s}_i$ ($0\le i\le \ell-1$) backward recursively as follows. Assume ${\bf s}^{(i)}$ is defined for some $i$ satisfying $1\le i\le \ell$. Define $w_i\in[1,n]$ to be the smallest adjustable position, and the corresponding new GCS by ${\bf s}^{(i-1)}$. (It is clear that $w_1,\dots,w_\ell$ are mutually distinct, and the set $\{w_1,\dots,w_\ell\}=\{j\in Q_0\, |\, s_j=0\}$.)
\end{definition}

\begin{definition}\label{df:g}
(i) Define a function $g_Q: \{0,1\}^{n}\to \{0,\pm1\}^{n'}$ sending ${\bf s}$ to $(g_r)_{1\le r\le n'}$ where
{\small 
\begin{equation}\label{df:gr}
 g_r= \left\{\begin{array}{lr}
        \deg_{Q,{\bf s}}^{1\to}(r)+\deg_{Q,{\bf s}}^{0\leftarrow}(r)-1, \quad \text{if $r\in Q_0$, or $r$ and two vertices in $Q_0$ form a 3-cycle;}\\
        \deg_{Q,{\bf s}}^{1\to}(r)+\deg_{Q,{\bf s}}^{0\leftarrow}(r), \quad \text{otherwise.}\\
        \end{array}\right.
\end{equation} 
}
Here $\deg^{1\to}_{Q,{\bf s}}(i)$ is the number of 
arrows $j\to i$ in $Q'_1$ such that $j\in Q_0$ and $s_j=1$, and $\deg^{0\leftarrow}_{Q,{\bf s}}(i)$ is the number of 
arrows $j\leftarrow i$ in $Q'_1$ such that $j\in Q_0$ and $s_j=0$. (Remark \ref{rmk:mi-1to1}(2) explains why $g_r\in\{0,\pm1\}$.) 

(ii) Define $m_i=(m_{i,r}):=g_Q({\bf s}^{(i)})\in \{0,\pm1\}^{n'}$ for $0\le i\le \ell$. 

\end{definition}

\begin{remark}\label{rmk:mi-1to1}
(1) Note that  ${\bf s}^{(0)}=[1,\dots,1]\in\{0,1\}^n$. We claim that $m_0=(g_r)$ coincides with the definition given in Lemma \ref{lem:g}. Indeed, note that $\deg_{Q,{\bf s}}^{0\leftarrow}(r)=0$. If $r\in Q_0$, nothing needs to be proved. We assume $r\notin Q_0$ in the rest of the paragraph. If $(\deg_Q^+(r),\deg_Q^-(r))=(0,1)$, then we are in the second case of \eqref{df:gr}, $g_r=\deg_{Q,{\bf s}}^{1\to}(r)=\deg_Q^-(r)=1$; if $(\deg_Q^+(r),\deg_Q^-(r))=(0,0)$ or $(1,0)$, then we are in the second case of \eqref{df:gr}, thus $g_r=\deg_{Q,{\bf s}}^{1\to}(r)=\deg_Q^-(r)=0$;  if  $(\deg_Q^+(r),\deg_Q^-(r))=(1,1)$, then  we are in the first case of \eqref{df:gr}, thus $g_r=\deg_{Q,{\bf s}}^{1\to}(r)-1=\deg_Q^-(r)-1=0$.  

(2) It is easy to check that every coordinate of $m_i=(g_r)$ satisfies $-1\le g_r\le 1$. Indeed, $g_r\ge -1$ is obvious; to check $g_r\le 1$, note that every vertex $r$ is adjacent to at most two vertices in $Q_0$, so we only need to show that in the second case of \eqref{df:gr},  it is impossible to have $\deg_{Q,{\bf s}}^{1\to}(r)+\deg_{Q,{\bf s}}^{0\leftarrow}(r)\ge 2$. But this equality implies that $r$ is adjacent to at least two vertices in $Q_0$; it follows from the description of type A quivers that  $r$ and two vertices in $Q_0$ form a 3-cycle, which contradicts the condition of the second case.
\end{remark}

The following main theorem gives a bijective construction of $\vartheta_{\QQ,m_0}=x[{\bf a}_Q]$.

\begin{theorem}\label{MainThm2}
Assume the rank $n'$  of the cluster algebra is even, and   $\QQ=(q_1,q_2,\dots,q_{n'})$ such that 
\begin{equation}\label{conditionQ}
0< q_i\ll q_1\ll q_2\ll\cdots\ll q_{n}, \textrm{ for each $i=n+1,\dots n'$.}
\end{equation} 
(Here $x\ll y$ means $0<x/y\le \varepsilon$, where $\varepsilon>0$ is a fixed real number satisfying $(1+\varepsilon)^n<2$.)
Let $m_0$ be defined as in Lemma \ref{lem:g}. Then there is a bijection $\varphi$ between the set of GCS and the set of broken lines for $m_0$ with endpoint $\QQ$, such that each GCS ${\bf s}$ is sent to a broken line $\gamma=\varphi({\bf s})$ satisfying 
${\rm Mono}(\gamma)=z_{\bf s}$ (defined in \eqref{cluster variable GCS formula}):
using notation in Definition \ref{df:wisi}, the broken line $\gamma$ can be explicitly described as follows: 

{\rm(i)} it has $\ell+1$ domains of linearity $L_0,L_1,\dots,L_\ell$ (where $L_0$ is unbounded); 

{\rm(ii)} it bends from the domain of linearity $L_{i-1}$ to $L_{i}$ at a point on the wall $\mathfrak{d}_{w_i}$;

{\rm(iii)} $\gamma'(t)=-m_i$ for $t\in L_i$;

{\rm(iv)} the monomial attached to $L_i$ is $z^{m_i}$. 
\end{theorem}

\subsection{Proof of Theorem \ref{MainThm2}}\label{section:proof}
\subsubsection{We show that (i)--(iv) determine a valid broken line.} The setting of \cite{GHKK} for a  type A cluster algebra is specialized as follows: the lattice $N_{\rm uf}=N=\mathbb{Z}^{n'}$ is equipped with a basis $e_1,\dots,e_{n'}$ and with a skew-symmetric bilinear form
$\{\cdot,\cdot\}: N\times N\to \mathbb{Q}$ satisfying
$\{e_i,e_j\}=\epsilon_{ij}=-b_{ij}$. (Since type A is skew-symmetric,  all the multipliers $d_1=\cdots=d_{n'}=1$ as noted in \cite[p114]{GHKK}). The dual lattice $M$ has a basis $f_1,\dots,f_{n'}$ dual to $e_1,\dots,e_{n'}$. The vector $v_j$ defined in \cite[p29]{GHKK} is
$$v_i:=\{e_i,\cdot\}=\sum_{j=1}^{n'} b_{ji}f_j.$$

\noindent\underline{Step 1}: check $z^{m_{i-1}}f_{\mathfrak{d}_{w_i}}^{\langle n_0,m_{i-1}\rangle}$ contains a term $z^{m_i}$, for $1\le i\le \ell$. 
(Here $n_0$ is the primitive vector annihilating the tangent space to $\mathfrak{d}_{w_i}$ and that $
\langle n_0,m_{i-1}\rangle$ is positive.) Since $n_0=\pm e_{w_i}$, and by Remark \ref{rmk:mi-1to1}(2), all coordinates of $m_{i-1}$ take value in $\{-1,0,1\}$, we must have $\langle n_0,m_{i-1}\rangle=1$. Meanwhile, by the initial scattering diagram described \cite[p31]{GHKK}, 
$f_{\mathfrak{d}_{w_i}}=1+z^{v_{w_i}}$. So
$$z^{m_{i-1}}f_{\mathfrak{d}_{w_i}}^{\langle n_0,m_{i-1}\rangle}=
z^{m_{i-1}}+z^{m_{i-1}+v_{w_i}},$$  
and we are left to show that $v_{w_i}=m_i-m_{i-1}$, or equivalently $b_{rw_i}=m_{i,r}-m_{i-1,r}$, or equivalently
\begin{equation}\label{b=X+Y}
 b_{rw_i}=X+Y
\end{equation}
where $X=\deg_{Q,{\bf s}^{(i)}}^{1\to}(r)-\deg_{Q,{\bf s}^{(i-1)}}^{1\to}(r)$, 
$Y=\deg_{Q,{\bf s}^{(i-1)}}^{0\leftarrow}(r)-\deg_{Q,{\bf s}^{(r-1)}}^{0\leftarrow}(r)$. Since 
${\bf s}^{(i)}$ and ${\bf s}^{(i-1)}$ only differ in the $w_i$-th coordinate (where the former has coordinate $0$ and the latter has coordinate $1$), we have  
$$X= \left\{\begin{array}{lr}
        -1, \quad \text{if there is an arrow $w_i\to r$;}\\
        0, \quad \text{otherwise.}\\
        \end{array}\right.
\quad \textrm{ and }
Y= \left\{\begin{array}{lr}
        1, \quad \text{if there is an arrow $r\to w_i$;}\\
        0, \quad \text{otherwise.}\\
        \end{array}\right.
$$
Then \eqref{b=X+Y} can be shown case by case: (1) if $b_{rw_i}=0$, then $r=w_i$ or $r$ is not adjacent to $w_i$, in either situation we have $X=Y=0$, thus \eqref{b=X+Y} holds; (2) if $b_{rw_i}=1$: then $X=0$, $Y=1$, and  \eqref{b=X+Y} holds; (3) if $b_{rw_i}=-1$, then $X=-1$, $Y=0$  and \eqref{b=X+Y} still holds. 

\medskip

\noindent\underline{Step 2}: denoting by $\QQ_i\in \mathfrak{d}_{w_i}$ the point where $\gamma$ bends from the domain of linearity $L_{i-1}$ to $L_{i}$, check that $\QQ_{i}-\QQ_{i+1}\in\mathbb{R}^+\,  m_i$ for $i=1,\dots,\ell$ (assume $\QQ_{\ell+1}=\QQ$). 

Note that $\QQ_i=(q^{(i)}_j)$ are determined by the following conditions:
$$\QQ_{\ell+1}=\QQ,\quad q^{(i)}_{w_i}=0,\quad \QQ_{i}-\QQ_{i+1}\in\mathbb{R}\,  m_i.$$

For convenience, we introduce the following definition: for $x, y, r\in\mathbb{R}$, $y>0$ and $1\le r <2$, define
$$x\approx_r y \Longleftrightarrow 2-r\le \frac{x}{y}\le r\;\; \Big(\Longleftrightarrow \Big|\frac{x}{y}-1\Big|\le r-1\Big).$$
Note that $x\approx_r y$ implies $x>0$. 


\begin{lemma}\label{Qi-Qi+1}
{\rm (i)} The $w_i$-th coordinate of $m_i$ is $-1$.

{\rm (ii)}  $\QQ_{i}-\QQ_{i+1}=q^{(i+1)}_{w_i} m_i$.

{\rm (iii)} $q^{(i')}_{w_i} \approx_{(1+\varepsilon)^{\ell+1-i'}} q_{w_i}$ for $1\le i<i'\le \ell+1$. As a consequence, for all $1\le i\le \ell$, we must have $q^{(i+1)}_{w_i}>0$, hence  $\QQ_{i}-\QQ_{i+1}\in \mathbb{R}^+ m_i$.
\end{lemma}
\begin{proof}
(i) By \eqref{df:gr},
$$m_{i,w_i}= 
        \deg_{Q,{\bf s}^{(i)}}^{1\to}(w_i)+\deg_{Q,{\bf s}^{(i)}}^{0\leftarrow}(w_i)-1$$
So to show $m_{i,w_i}=-1$, it is equivalent to show that the above degrees are both 0.

Since ${\bf s}^{(i)}_{w_i}=0$, we must have ${\bf s}^{(i)}_{j}=0$ for any arrow $j\to w_i$, otherwise it contradicts the assumption that ${\bf s}^{(i)}$ is a GCS.  Thus $\deg_{Q,{\bf s}^{(i)}}^{1\to}(w_i)=0$. 

Similarly, since  ${\bf s}^{(i-1)}_{w_i}=1$, we must have ${\bf s}^{(i-1)}_{j}=1$ for any arrow $w_i\to j$, otherwise it contradicts the assumption that ${\bf s}^{(i-1)}$ is a GCS. Thus  $\deg_{Q,{\bf s}^{(i-1)}}^{0\leftarrow}(w_i)=0$. Since ${\bf s}^{(i-1)}$ and ${\bf s}^{(i)}$ only differ in the $w_i$-th coordinate, 
we have $\deg_{Q,{\bf s}^{(i)}}^{0\leftarrow}(w_i)=\deg_{Q,{\bf s}^{(i-1)}}^{0\leftarrow}(w_i)=0$. 

(ii) Assume that $\QQ_i-\QQ_{i+1}=\lambda\, m_i$. To determine $\lambda$, it suffices to consider the $w_i$-th coordinate on both sides:
$$q^{(i)}_{w_i}-q^{(i+1)}_{w_i}=\lambda\, m_{i,w_i}=-\lambda$$
Then (ii) holds since $q^{(i)}_{w_i}=0$.

(iii) We prove it by fixing $i$ and using downward induction on $i'$. For $i'=\ell+1$, $q^{(i')}_{w_i}=q_{w_i}$, hence the statement holds. For $i'<\ell+1$, using (ii) we have
\begin{equation}\label{qq}
q^{(i')}_{w_i}-q^{(i'+1)}_{w_i}=q^{(i'+1)}_{w_{i'}}m_{i',w_i}.
\end{equation}
Note that $w_{i'}\neq w_i$ since $i'\neq i$. We argue in two cases:

Case 1: $w_{i'}<w_i$. We have 
$q^{(i'+1)}_{w_{i}}\approx_{(1+\varepsilon)^{\ell-i'}} q_{w_{i}}$
and 
$q^{(i'+1)}_{w_{i'}}\approx_{(1+\varepsilon)^{\ell-i'}} q_{w_{i'}}$
 by inductive assumption. Since $m_{i',w_i}\in\{0,\pm1\}$, \eqref{qq} implies
$$\aligned
 &\Big|\frac{q^{(i')}_{w_i}}{q_{w_i}}-1\Big |
=
\Big| \frac{q^{(i'+1)}_{w_i} + q^{(i'+1)}_{w_{i'}}m_{i',w_i}- q_{w_i} }{q_{w_i}} \Big|
\le 
\Big| \frac{q^{(i'+1)}_{w_i} - q_{w_i} }{q_{w_i}} \Big|+
\Big| \frac{q^{(i'+1)}_{w_{i'}}}{q_{w_i}} \Big|\\
&\quad 
=
\Big| \frac{q^{(i'+1)}_{w_i}}{q_{w_i}}-1 \Big|+
\Big| \frac{q^{(i'+1)}_{w_{i'}}}{q_{w_{i'}}} \Big|\cdot \Big| \frac{q_{w_{i'}}}{q_{w_{i}}} \Big|
\le (1+\varepsilon)^{\ell-i'}-1+ (1+\varepsilon)^{\ell-i'}\varepsilon=(1+\varepsilon)^{\ell+1-i'}-1.
\endaligned$$

Case 2: $w_{i'}>w_i$. We shall show that $m_{i',w_i}=0$ (which implies $q^{(i')}_{w_i}=q^{(i'+1)}_{w_i}\approx_{\ell+1-(i'+1)} q_{w_i}$, therefore $q^{(i')}_{w_i}\approx_{\ell+1-i'} q_{w_i}$,). It suffices to show that the two degrees in the following expression are $0$ and $1$, respectively:
$$m_{i',w_i}= 
        \deg_{Q,{\bf s}^{(i')}}^{1\to}(w_i)+\deg_{Q,{\bf s}^{(i')}}^{0\leftarrow}(w_i)-1$$

Since ${\bf s}^{(i')}$ is  GCS and ${\bf s}^{(i')}_{w_i}=0$ by the construction of ${\bf s}^{(i')}$, there is no arrow $j\to w_i$ in $Q$ satisfying ${\bf s}^{(i')}_{j}=1$. Thus $ \deg_{Q,{\bf s}^{(i')}}^{1\to}(w_i)=0$.

Next, we prove $\deg_{Q,{\bf s}^{(i')}}^{0\leftarrow}(w_i)=1$ by contradiction:

If $\deg_{Q,{\bf s}^{(i')}}^{0\leftarrow}(w_i)=0$, then there is no arrow $w_i\to j$ in $Q$ satisfying ${\bf s}^{(i')}_{j}=0$. But then $w_i$ is also an adjustable position in ${\bf s}^{(i')}$, which contradicts the property of being ``smallest'' in the definition of $w_{i'}$ (Definition \ref{df:wisi}).

If $\deg_{Q,{\bf s}^{(i')}}^{0\leftarrow}(w_i)\ge 2$, then by our assumption on $Q$ (see \eqref{labelQ1ton}), there must be two arrows  $w_i\to (w_i-1)$ and $w_i\to (w_i+1)$ in $Q$ satisfying ${\bf s}^{(i')}_{w_i-1}={\bf s}^{(i')}_{w_i+1}=0$. Consider the longest path starting from $w_i$ and going left:
 $j\leftarrow (j+1)\leftarrow\cdots\leftarrow (w_i-1)\leftarrow w_i$ in $Q^{\;{\bf s}^{(i')}}$ (defined in Remark \ref{rmk:adjustableposition}). Then $j$ is a sink in $Q^{\;{\bf s}^{(i')}}$, hence adjustable and satisfying $j<w_i<w_{i'}$, again contradicts the property of being ``smallest'' in the definition of $w_{i'}$.
\end{proof}

\subsubsection{The map $\varphi$ is bijective}\label{subsection:bijective}

It is easy to see that $\varphi$ injective, since different $GCS$ determine different sets of $w_i$ ( i.e., the sets of walls where the broken line bend), thus determine different broken lines. 
Thus to show the bijectivity of $\varphi$, it suffices to show that the number of GCC is not less than the number of broken lines. To show the latter,  we use \cite[Theorem 7.5 (4)]{GHKK} which asserts that the cluster variable $x[{\bf a}_Q]$ is equal to the theta function $\vartheta_{\QQ,m_0}$. Since the number of GCS is equal to $x[{\bf a}_Q]|_{x_1=\cdots=x_{n'}=1}$, on the other hand each broken line contributes at least 1 to $x[{\bf a}_Q]|_{x_1=\cdots=x_{n'}=1}$, we conclude that the number of GCS is not less than the number of broken lines. Therefore $\varphi$ is bijective. 
This completes the proof of Theorem \ref{MainThm2}.

Note that the above argument implies that each broken line that contributes to $\vartheta_{\QQ,m_0}$ will not bend on any outgoing walls.

\subsection{A bijection between GCSs and broken lines (in the general case)}
To extend Theorem \ref{MainThm2} to odd ranks, we need to consider principal coefficients. We denote by $\tilde{x}[{\bf a}]$ the cluster variable in $\mathcal{A}_{\rm prin}$ that corresponds to the cluster variable  $x[{\bf a}]$ in $\mathcal{A}$.

\begin{lemma}\label{lemma:g-vector principal}
Let $\mathcal{A}$ be a (coefficient-free) cluster algebra of rank $n$,  and $\mathcal{A}_{\rm prin}$ the corresponding cluster algebra with principal coefficients. 
Let $x[{\bf a}]$ be a cluster variable in $\mathcal{A}$ with g-vector ${\bf g}\in\mathbb{Z}^n$. Then the corresponding cluster variable $\tilde{x}[{\bf a}]$ in $\mathcal{A}_{\rm prin}$ has g-vector $$\begin{bmatrix}{\bf g}\\0\end{bmatrix}\in\mathbb{Z}^n\times 0^n\subset \mathbb{Z}^{2n}.$$ 
\end{lemma}
\begin{proof}
Recall that the g-vector of $x[{\bf a}]\in\mathcal{A}$ is the multidegree of the corresponding cluster variable in principal coefficients. More precisely, start with matrix $\tilde{B}_{t_0}=\begin{bmatrix}
B\\ I
\end{bmatrix}
$, and for each mutation $t\stackrel{k}{\textrm{------}}t'$, $\tilde{B}_t$ and $\tilde{B}_{t'}$ are related by the rule
\begin{equation}\label{mutate b}
b^{t'}_{ij}=\begin{cases}
-b^t_{ij},  &\textrm{ if $i=k$ or $j=k$;}\\
b^t_{ij}+{\rm sgn }(b^t_{ik})[b^t_{ik}b^t_{kj}]_+, &\textrm{ otherwise.}
\end{cases}
\end{equation} 
the new cluster variable is determined by
\begin{equation}\label{mutation x}
x_{k;t'}x_{k;t}=\prod_{i=1}^{2n}x_{i;t}^{[b_{ik}^t]_+}+\prod_{i=1}^{2n}x_{i;t}^{[-b_{ik}^t]_+}.
\end{equation}
Each cluster variable is homogeneous with the assignment that, for $1\le i\le n$, $\deg(x_i)={\bf e}_i\in\mathbb{Z}^n$, $\deg(x_{n+i})=-{\bf b}_i\in\mathbb{Z}^n$, where ${\bf b}_i$ is the $i$-th column of $B$. This multidegree is the g-vector of the cluster variable. 

The g-vector of $\tilde{x}[{\bf a}]\in\mathcal{A}_{\rm prin}$ is defined similarly as above, with $B$ and $\tilde{B}_{t_0}$ being replaced by 
$$\bar{B}=\begin{bmatrix}
B&-I\\I&0
\end{bmatrix}\;
\textrm{ and } 
 \tilde{\bar{B}}_{t_0}=\begin{bmatrix}
B&-I\\ I&0\\ I&0\\ 0&I
\end{bmatrix}, \textrm{ respectively}. 
$$
Using \eqref{mutate b}, it can be proved by a simple induction that,
$$
\tilde{\bar{B}}_{t}=\begin{bmatrix}
B_t&-C_t\\ C_t&D_t\\ C_t&D_t\\ 0&I
\end{bmatrix}, \textrm{ if }
\tilde{B}_{t}=\begin{bmatrix}
B_t\\ C_t
\end{bmatrix}. 
$$ 
(In fact, $D_t=0$ if the sign-coherence conjecture is true \cite[Conjecture 8.8]{Reading}; in particular, it is true if $B$ is skew-symmetric, but we do not need this fact in this paper.)
Using \eqref{mutation x}, it can be proved by a simple induction that 
$\tilde{x}[{\bf a}]$ can be obtained from $x[{\bf a}]$ by the substitution $x_i\mapsto\tilde{x}_i$ and $x_{n+i}\mapsto\tilde{x}_{n+i}\tilde{x}_{2n+i}$, for $1\le i\le n$. Since 
$$
\aligned
&\deg \tilde{x}_i=\begin{bmatrix}{\bf e}_i\\0\end{bmatrix}=\begin{bmatrix}\deg x_i\\0\end{bmatrix}, \\
&\deg
(\tilde{x}_{n+i}\tilde{x}_{2n+i})=\begin{bmatrix}0\\ {\bf e}_i\end{bmatrix}-(\textrm{the $i$-th column of $\bar{B}$})=\begin{bmatrix}0\\ {\bf e}_i\end{bmatrix}-\begin{bmatrix}{\bf b}_i\\{\bf e}_i\end{bmatrix}
=\begin{bmatrix}-{\bf b}_i\\0\end{bmatrix}
=\begin{bmatrix}\deg x_{n+i}\\0\end{bmatrix},
\endaligned$$ 
the multidegree (i.e., the g-vector) of  $\tilde{x}[{\bf a}]$ must be $\begin{bmatrix}{\bf g}\\0\end{bmatrix}$ where $\bf g$ is the multidegree (i.e., the g-vector) of  $x[{\bf a}]$. 
\end{proof}

In the rest we show that Theorem \ref{MainThm2} can be adapted to $\mathcal{A}_{\rm prin}$.  

\begin{theorem}\label{MainThm2'}
Assume $\tilde{\QQ}=(q_1,q_2,\dots,q_{2n'})$ such that 
\begin{equation*}
0< q_i\ll q_1\ll q_2\ll\cdots\ll q_{n}, \textrm{ for each $i=n+1,\dots n'$.}
\end{equation*} 
(Here $x\ll y$ means $0<x/y\le \varepsilon$, where $\varepsilon>0$ is a fixed real number satisfying $(1+\varepsilon)^n<2$.
There is no condition on $q_{n'+1},\dots,q_{2n'}$.)
Let $$\tilde{m}_0=\begin{bmatrix}m_0\\0\end{bmatrix},$$ where $m_0$ is defined in Lemma \ref{lem:g}.
Then there is a bijection $\varphi$ between the set of GCS and the set of broken lines for $\tilde{m}_0$ with endpoints $\tilde{\QQ}$, such that each GCS ${\bf s}$ is sent to a broken line $\gamma=\varphi({\bf s})$ satisfying 
${\rm Mono}(\gamma)\big|_{x_{n'+1}=\cdots=x_{2n'}=1}=z_{\bf s}$ (defined in \eqref{cluster variable GCS formula}):
using notation in Definition \ref{df:wisi}, the broken line $\gamma$ can be explicitly described as follows: 

{\rm(i)} it has $\ell+1$ domains of linearity $L_0,L_1,\dots,L_\ell$ ($L_0$ is unbounded); 

{\rm(ii)} it bends from the domain of linearity $L_{i-1}$ to $L_{i}$ at a point on the wall $\mathfrak{d}_{w_i}$;

{\rm(iii)} $\gamma'(t)=-\tilde{m}_i$ for $t\in L_i$, where $\tilde{m}_i=\begin{bmatrix} m_i\\ \sum_{r=1}^i {\bf e}_{w_r}\end{bmatrix}$. 

{\rm(iv)} the monomial attached to $L_i$ is $z^{\tilde{m}_i}$. 
\end{theorem}

\begin{proof}
Since $\bar{B}$ is full rank, the Fundamental Assumption in \cite{GHKK} is satisfied, thus the proof in \S\ref{section:proof}  works with little change; we point out the nontrivial revision change.

Because of Lemma \ref{lemma:g-vector principal}, the $g$-vector of $\tilde{x}[{\bf a}]$ is indeed $\tilde{m}_0$. Thus $\tilde{x}[{\bf a}]$ is equal to the theta function $\vartheta_{\tilde{\QQ},\tilde{m}_0}$ by  \cite[Theorem 7.5 (4)]{GHKK}. 
 
In Step 1, $n_0$ is replaced by $\tilde{n}_0=\begin{bmatrix}n\\0\end{bmatrix}\in\mathbb{R}^{2n'}$. Since the first $n'$ coordinates of $\tilde{m}_{i-1}$ form the vector $m_{i-1}$, we have 
$\langle \tilde{n}_0,\tilde{m}_i\rangle=\langle n_0,m_i\rangle=1$. The function attached to the wall $\mathfrak{d}_i$ is 
$\tilde{f}_{\mathfrak{d}_i}=1+z^{\tilde{v}_i}$ where ${\tilde{v}_i}=\begin{bmatrix}v_i\\ {\rm e}_i\end{bmatrix}$. Thus
$$\tilde{v}_{w_i}=\begin{bmatrix}v_{w_i}\\{\rm e}_{w_i} \end{bmatrix}=\begin{bmatrix}m_i-m_{i-1}\\{\rm e}_{w_i} \end{bmatrix}=\tilde{m}_i-\tilde{m}_{i-1}.$$ 

In Step 2, since $\pi(\tilde{\QQ}_i)=\QQ_i$, we have $\tilde{\QQ}_i-\tilde{\QQ}_{i+1}=q^{(i+1)}_{w_i} \tilde{m}_i$ with the same $q^{(i+1)}_{w_i}>0$ as in Lemma  \ref{Qi-Qi+1}. 
\end{proof}

\section{Examples}
In this section, we give several examples to illustrate the computation of cluster variables and cluster monomials using methods introduced in previous sections.
\begin{example}
We compute $x[2,2,2]$, the cluster monomial with \textbf{d}-vector $(2,2,2)$ for the quiver $Q=1\to 2\to 3\to 1$ (a 3-cycle).

-- Using formula  \eqref{formula1}: choose $i_0=1$. Then $d(1)=0$, $d(2)=1$, $d(3)=2$, $\sigma_{123}=\sigma_{231}=\sigma_{312}=1$. A GCS ${\bf s}=({\bf s}_1,{\bf s}_2,{\bf s}_3)$ satisfies
$$(s_{1,1},s_{2,1}), (s_{3,2},s_{1,2}), \textrm{ and } (s_{2,2},s_{3,1})\neq(1,0).$$
A tedious computation using \eqref{formula1} gives $x[2,2,2]=x_1^{-2}x_2^{-2}x_3^{-2}(x_1+x_2+x_3)^3$.  (But at  least it is easy to see that there are 27 terms; indeed, since each pair has 3 choices $(0,0),(1,1),(0,1)$, the total number of GCSs is $3\times 3\times 3=27$.)

-- Using formula  \eqref{formula2}: there are three Dyck paths of size $2\times 2$ corresponding to the three arrows $1\to 2$, $2\to 3$, $3\to 1$:
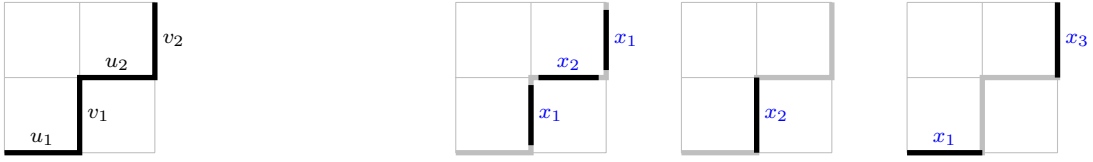
\begin{figure}[h!]
\begin{tikzpicture}[MyPic]
\tikzstyle{every node}=[font=\tiny]
\tsomepoints
\draw (0,0) grid (2,2);
\draw [SoThick]
(0,0) -- (1,0)  node[midway,above] {$u_1$} --
(1,0) -- (1,1)  node[midway,right] {$v_1$} --
(1,1) -- (2,1)  node[midway,above] {$u_2$} --
(2,1) -- (2,2)  node[midway,right] {$v_2$};
	
\begin{scope}[shift={(6,0)}]
\draw (0,0) grid (2,2);
\draw[ line width=2pt]
(0,0) -- (1,0)  node[midway,above] {} --
(1,0) -- (1,1)  node[midway,right] {\color{blue} $x_1$} --
(1,1) -- (2,1)  node[midway,above] {\color{blue} $x_2$} --
(2,1) -- (2,2)  node[midway,right] {\color{blue} $x_1$};
\draw [SoThick]
(1,0.1) -- (1,.9)
(1.1,1) -- (1.9,1)
(2,1.1) -- (2,1.9);
\end{scope}
	
\begin{scope}[shift={(9,0)}]
\draw (0,0) grid (2,2);
\draw[ line width=2pt]
(0,0) -- (1,0)  node[midway] {} --
(1,0) -- (1,1)  node[midway,right] {\color{blue} $x_2$} --
(1,1) -- (2,1)  node[midway] {} --
(2,1) -- (2,2)  node[midway] {};
\draw [SoThick] (1,0) -- (1,1);
\end{scope}
	
\begin{scope}[shift={(12,0)}]
\draw (0,0) grid (2,2);
\draw[ line width=2pt]
(0,0) -- (1,0)  node[midway,above] {\color{blue} $x_1$} --
(1,0) -- (1,1)  node[midway] {} --
(1,1) -- (2,1)  node[midway] {} --
(2,1) -- (2,2)  node[midway,right] {\color{blue} $x_3$};
\draw [SoThick]
(0,0) -- (1,0)
(2,1) -- (2,2);
\end{scope}
\end{tikzpicture}
\caption{Left: the $2\times2$ maximal Dyck path\quad  \quad Right: An example of GCC.}
\label{fig 9}
\end{figure}
such that (i) we do not choose both $u_1$ and $v_1$ in each Dyck path, and (ii) we choose $v_r$ in the $i$-th Dyck path if and only if we do not choose $u_{3-r}$ in the $(i+1)$-th Dyck path for $r=1,2$ (by convention, the 4th Dyck path is the 1st one). In the example of GCC in Figure \ref{fig 9} (Right), the corresponding product
$$\left(\prod_{i\to j}  x_{i}^{\big|S^{(i,j)}_2\big|} x_{j}^{\big|S^{(i,j)}_1\big|}\right)\cdot \prod_{i\to j\to k\to i} x_i^{-\sigma_{jki}}=(x_1^3x_2^2x_3)\cdot x_1^{-1}x_2^{-1}x_3^{-1}=x_1^2x_2.$$
Computing all possible GCCs gives $x[2,2,2]=x_1^{-2}x_2^{-2}x_3^{-2}(x_1+x_2+x_3)^3$.

-- Using formula  \eqref{formula3}: first observe that there are 3 pipelines as shown in Figure \ref{pipelines}.
\begin{figure}[h!]
  \centering
    \includegraphics[width=.2\textwidth]{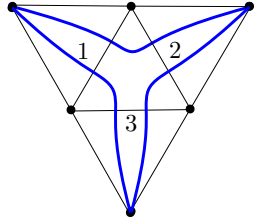}
    \caption{Pipelines}
    \label{pipelines}
\end{figure}

According to \S\ref{subsection3.3}, $x[2,2,2]=x[1,1,0]x[0,1,1]x[1,0,1]$. Now we compute $x[1,1,0]$ using Theorem \ref{MainThm}. The pair $(|S_{1,1}|,|S_{1,2}|)$ can be $(1,0), (0,1), (0,0)$. Correspondingly, we have $(y_1,y_0,y_2)=(x_1,x_{1,0},x_{2,0}), (x_2,x_{1,1},x_{2,1}), (x_{1,2},x_{1,1},x_{2,0})$.
Then
$$\aligned
x[1,1,0]&=(x_1^{-1}x_2^{-1})\sum y_0y_1y_2=(x_1^{-1}x_2^{-1})(x_1x_{1,0}x_{2,0} + x_2x_{1,1}x_{2,1} + x_{1,2}x_{1,1}x_{2,0})\\
&=(x_1^{-1}x_2^{-1})(x_1+ x_2 + x_3)
\endaligned$$
where the last equality is obtained by substituting $x_{1,2}=x_3$, and $x_{1,0}=x_{1,1}=x_{2,0}=x_{2,1}=1$. Similarly, $x[0,1,1]=(x_2^{-1}x_3^{-1})(x_1+ x_2 + x_3)$, $x[1,0,1]=(x_1^{-1}x_3^{-1})(x_1+ x_2 + x_3)$. Thus $x[2,2,2]=x_1^{-2}x_2^{-2}x_3^{-2}(x_1+x_2+x_3)^3$.
\end{example}

\begin{example}\label{example62} We compute some cluster variables of  $\A(\Qcal)$ where $\Qcal$ is the following type $A$ quiver:
\begin{center}
\begin{tikzpicture} [scale=0.8]
\node (v2) at (0:1) {2};
\node (v1) at (120:1) {1};
\node (v5) at (240:1) {5};
\path (30:1.732)++(1,0) node (v6) {6};
\path (-30:1.732)++(1,0) node (v3) {3};
\path (0:1.732)++(2.5,-0.866) node (v4) {4};
\path (30:1.732)++(2.732,0) node (v7) {7};
		
\draw [->] (v1) to (v2);
\draw [->] (v2) to (v5);
\draw [->] (v5) to (v1);
\draw [->] (v2) to (v6);
\draw [->] (v6) to (v3);
\draw [->] (v3) to (v2);
\draw [->] (v3) to (v4);
\draw [->] (v6) to (v7);
\end{tikzpicture}
\end{center}
	
As observed in Remark \ref{rmk 4.7}, the set of non-initial cluster variables is in one-to-one correspondence with the set of \textbf{d}-vectors ${\bf a}=(a_1,\dots,a_n)$, where $a_i\in\{0,1\}$ and $\{i\, |\, a_i=1\}$ is the vertex set of a linear full subquiver of $\Qcal$. 
So we can compute all cluster variables using Theorem \ref{MainThm}. 

For example, we consider the \textbf{d}-vector $\bfa=(1,1,1,0,0,0,0)$. Then the subset of vertices $\{i|a_i=1\}$ is equal to the set of vertices $Q_0$ of the full linear subquiver $Q=1\to2\leftarrow3$. It is a subquiver of an extended linear subquiver $P$ and after completing $P$, we get a completely extended linear quiver $Q'$ as shown in Figure \ref{QQ'}.
	
\begin{figure}[h!]
\begin{tabular}{lll}
\begin{tikzpicture} [scale=0.8]
\node (v1) at (0:1) [red] {1};
\path (-60:1.732)++(1,0) node (v5) {5};
\path (0:1.732)++(1,0) node (v2) [red] {2};
\path (0:1.732)++(2.732,0) node (v3) [red] {3};
\path (60:1.732)++(2.732,0) node (v6) {6};
\path (-30:1.732)++(4.464,0) node (v4) {4};
			
\draw [->,red] (v1) to (v2);
\draw [->] (v2) to (v5);
\draw [->] (v5) to (v1);
\draw [->,red] (v3) to (v2);
\draw [->] (v2) to (v6);
\draw [->] (v6) to (v3);
\draw [->] (v3) to (v4);
\end{tikzpicture}
&
\begin{tikzpicture}[scale=0.8]
\node (v1) at (0:0) {};
\node (v2) at (0:2) {};
\node (v3) at (-60:1.732) {};
		
\draw [->] (v1) to (v2);
\end{tikzpicture}
&
\begin{tikzpicture} [scale=0.8]
\node (v1) at (0:1) [red] {1};
\node (v8) at (120:1) [blue] {8};
\node (v9) at (240:1) [blue] {9};
\path (-60:1.732)++(1,0) node (v5) {5};
\path (0:1.732)++(1,0) node (v2) [red] {2};
\path (0:1.732)++(2.732,0) node (v3) [red] {3};
\path (60:1.732)++(2.732,0) node (v6) {6};
\path (-30:1.732)++(4.464,0) node (v4) {4};
\path (30:1.732)++(4.464,0) node (v10) [blue] {10};
		
\draw [->,blue] (v1) to (v8);
\draw [->,blue] (v8) to (v9);
\draw [->,blue] (v9) to (v1);
\draw [->,red] (v1) to (v2);
\draw [->] (v2) to (v5);
\draw [->] (v5) to (v1);
\draw [->,red] (v3) to (v2);
\draw [->] (v2) to (v6);
\draw [->] (v6) to (v3);
\draw [->] (v3) to (v4);
\draw [blue,->] (v4) to (v10);
\draw [blue,->] (v10) to (v3);
\end{tikzpicture}
\end{tabular}
\caption{$P$ and its completed version $Q'$}
\label{QQ'}
\end{figure}
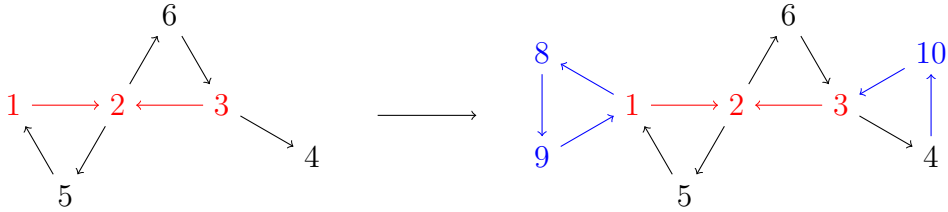
	
\noindent In $P$, we have $v_{1,2}=5$, $v_{2,3}=6$, $v_{3,0}=4$. Two 3-cycles are added and the new vertices are $v_{1,0}=8$, $v_{1,1}=9$ and $v_{3,1}=10$. All GCCs are described as follows.
	
\begin{center}
\begin{tabular}{lllll}
$\beta_1$
&
=
&
\begin{tabular}{l}
\begin{tikzpicture}[MyPic]
\tsomepoints
			
\draw (P00) -- (P10) node[midway,below] {$x_2$};
\draw (P10) -- (P11) node[midway,left] {$x_1$};
			
\draw (P20) -- (P21) node[midway,right] {$x_3$};
\draw (P20) -- (P30) node[midway,below] {$x_2$};
\end{tikzpicture}
\end{tabular}
&
,
&
$x(\beta_1)=(x_5x_6x_9x_{10})/(x_1x_2x_3)$
\end{tabular}\\
	
\begin{tabular}{lllll}
$\beta_2$
&
=
&
\begin{tabular}{l}
\begin{tikzpicture}[MyPic]
\tsomepoints
				
\draw (P00) -- (P10) node[midway,below] {$x_2$};
\draw (P10) -- (P11) node[midway,left] {$x_1$};
			
\draw (P20) -- (P21) node[midway,right] {$x_3$};
\draw [SoThick] (P20) -- (P30) node[midway,below] {$x_2$};
\end{tikzpicture}
\end{tabular}
&
,
&
$x(\beta_2)=(x_2x_4x_5x_9)/(x_1x_2x_3)$
\end{tabular}\\
	
\begin{tabular}{lllll}
$\beta_3$
&
=
&
\begin{tabular}{l}
\begin{tikzpicture}[MyPic]
\tsomepoints
			
\draw [SoThick] (P00) -- (P10) node[midway,below] {$x_2$};
\draw (P10) -- (P11) node[midway,left] {$x_1$};
			
\draw (P20) -- (P21) node[midway,right] {$x_3$};
\draw (P20) -- (P30) node[midway,below] {$x_2$};
\end{tikzpicture}
\end{tabular}
&
,
&
$x(\beta_3)=(x_2x_6x_8x_{10})/(x_1x_2x_3)$
\end{tabular}\\
	
\begin{tabular}{lllll}
$\beta_4$
&
=
&
\begin{tabular}{l}
\begin{tikzpicture}[MyPic]
\tsomepoints
			
\draw [SoThick] (P00) -- (P10) node[midway,below] {$x_2$};
\draw (P10) -- (P11) node[midway,left] {$x_1$};
			
\draw (P20) -- (P21) node[midway,right] {$x_3$};
\draw [SoThick] (P20) -- (P30) node[midway,below] {$x_2$};
\end{tikzpicture}
\end{tabular}
&
,
&
$x(\beta_4)=(x_2^2x_4x_8)/(x_1x_2x_3)$
\end{tabular}\\
	
\begin{tabular}{lllll}
$\beta_5$
&
=
&
\begin{tabular}{l}
\begin{tikzpicture}[MyPic]
\tsomepoints
			
\draw (P00) -- (P10) node[midway,below] {$x_2$};
\draw [SoThick] (P10) -- (P11) node[midway,left] {$x_1$};
			
\draw [SoThick] (P20) -- (P21) node[midway,right] {$x_3$};
\draw (P20) -- (P30) node[midway,below] {$x_2$};
\end{tikzpicture}
\end{tabular}
&
,
&
$x(\beta_5)=(x_1x_3x_9x_{10})/(x_1x_2x_3)$
\end{tabular}
\end{center}
	
\noindent The cluster variable with $\bfd$-vector $(1,1,1,0,0,0,0)$ is
$$x[1,1,1,0,0,0,0] = \sum_{i=1}^5 x(\beta_i) = \dfrac{x_5x_6x_9x_{10}+x_2x_4x_5x_9+x_2x_6x_8x_{10}+x_2^2x_4x_8+x_1x_3x_9x_{10}}{x_1x_2x_3}.$$
Setting $x_8=x_9=x_{10}=1$, we get the following cluster variable of $\A(\Qcal)$:
$$x[1,1,1,0,0,0,0]=\dfrac{x_5x_6+x_2x_4x_5+x_2x_6+x_2^2x_4+x_1x_3}{x_1x_2x_3}.$$
	
The table below shows some $\bfd$-vectors and their corresponding cluster variables of $\A(\Qcal)$.
	
\begin{center}
\begin{tabular}{clccl}
(1,0,0,0,0,0,0) & \qquad $\dfrac{x_2+x_5}{x_1}$ & \hspace{0.7cm} & (1,1,0,0,0,0,0) & \qquad $\dfrac{x_1x_3+x_2x_6+x_5x_6}{x_1x_2}$ \\[18pt]
			
(0,1,0,0,0,0,0) & \qquad $\dfrac{x_1x_3+x_5x_6}{x_2}$ & & (0,1,1,0,0,0,0) & \qquad $\dfrac{x_1x_3+x_2x_4x_5+x_5x_6}{x_2x_3}$ \\[18pt]
			
(0,0,1,0,0,0,0) & \qquad $\dfrac{x_2x_4+x_6}{x_3}$ & & (0,0,1,1,0,0,0) & \qquad $\dfrac{x_2x_4+x_3x_6+x_6}{x_3x_4}$ \\[18pt]
			
(0,0,0,1,0,0,0) & \qquad $\dfrac{1+x_3}{x_4}$ & & (1,0,0,0,1,0,0) & \qquad $\dfrac{x_1+x_2+x_5}{x_1x_5}$ \\[18pt]
			
& & & (0,0,1,0,0,1,0) & \qquad $\dfrac{x_2x_4+x_6+x_3x_4x_7}{x_3x_6}$ \\[18pt]
\hline\\
\end{tabular}
		
\begin{tabular}{cl}
(1,1,1,0,0,0,0) & \qquad $\dfrac{x_5x_6+x_2x_4x_5+x_2x_6+x_2^2x_4+x_1x_3}{x_1x_2x_3}$ \\[18pt]
		
(0,1,1,0,1,0,0) & \qquad $\dfrac{x_1x_3+x_2x_3+x_2x_4x_5+x_5x_6}{x_2x_3x_5}$ \\[18pt]
			
(0,1,1,1,0,0,0) & \qquad $\dfrac{x_1x_3+x_1x_3^2+x_2x_4x_5+x_5x_6+x_3x_5x_6}{x_2x_3x_4}$ \\[18pt]
			
(0,0,1,1,0,1,0) & \qquad $\dfrac{x_2x_4+x_6+x_3x_6+x_3x_4x_7}{x_3x_4x_6}$ \\[18pt]
			
(1,1,1,1,0,0,0) & \qquad $\dfrac{x_1x_3+x_1x_3^2+x_2^2x_4+x_2x_4x_5+x_2x_6+x_2x_3x_6+x_5	x_6+x_3x_5x_6}{x_1x_2x_3x_4}$
\end{tabular}
\end{center}
\end{example}

\medskip

\begin{example} Here we give an example to illustrate statements in \S\ref{S7}. 
Let $Q'$ be the following type A quiver:
\begin{center}
	\begin{tikzpicture} [scale=0.8]
	\node (v1) at (0:1) {1};
	\path (60:1.732)++(1,0) node (v4) {4};
	\path (0:1.732)++(1,0) node (v2) {2};
	\path (0:1.732)++(2.732,0) node (v3) {3};
		
	\draw [->] (v2) to (v1);
	\draw [->] (v1) to (v4);
	\draw [->] (v4) to (v2);
	\draw [->] (v2) to (v3);
	\end{tikzpicture}
\end{center}
Let $Q$ be the full linear subquiver $1 \leftarrow 2 \to 3$, and correspondingly $\bfa_Q=(1,1,1,0)$. 

There are 5 GCS ${\bf s}=(s_1,s_2,s_3)$:
$$(0,0,0), (1,0,0),(0,0,1),(1,0,1),(1,1,1)$$

By Lemma  \ref{lem:g}, $m_0=\begin{bmatrix} 0\\-1\\0\\0\end{bmatrix}$.

Consider ${\bf s}=(0,0,0)$. By Definition \ref{df:wisi}, $\ell=3$, 
$${\bf s}^{(3)}=(0,0,0),\; {\bf s}^{(2)}=(1,0,0), \;{\bf s}^{(1)}=(1,0,1), \;{\bf s}^{(0)}=(1,1,1), \;\;w_3=1, \;w_2=3,\; w_1=2.$$ 

By Definition \ref{df:g}, $m_0$ is as above,
$$m_1=\begin{bmatrix} -1\\-1\\-1\\1\end{bmatrix},\quad m_2=\begin{bmatrix} -1\\0\\-1\\1\end{bmatrix},\quad m_3=\begin{bmatrix} -1\\1\\-1\\0\end{bmatrix}.$$

The broken line $\gamma=\varphi({\bf s})$ described in Theorem \ref{MainThm2} has 4 domains of linearity $L_0,\dots,L_3$, bends on walls $\mathfrak{d}_2, \mathfrak{d}_3, \mathfrak{d}_1$ in that order. The direction vector of $\gamma$ on $L_i$ is $-m_i$ for $i=0,\dots,3$. The monomials attached to $L_0,\dots,L_3$ are $x_2^{-1}$, $x_1^{-1}x_2^{-1}x_3^{-1}x_4$,  $x_1^{-1}x_3^{-1}x_4$,  $x_1^{-1}x_2x_3^{-1}$, respectively. Thus ${\rm Mono }(\gamma)=x_1^{-1}x_2x_3^{-1}$. 

Follow the proof in Step 2 of \S\ref{section:proof}, we compute the coordinates of $\QQ_i$ as follows:
$$
\QQ_4=\QQ=\begin{bmatrix}q_1\\q_2\\q_3\\q_4\end{bmatrix}, \quad
\QQ_3=\QQ_4+q_1m_3=\begin{bmatrix} 0\\q_2+q_1\\q_3-q_1\\q_4\end{bmatrix}, \quad
\QQ_2=\QQ_3+(q_3-q_1)m_2=\begin{bmatrix}-q_3+q_1\\q_2+q_1\\ 0\\q_4+q_3-q_1\end{bmatrix}, \quad
$$
$$
\QQ_1=\QQ_2+(q_2+q_1)m_1=\begin{bmatrix}-q_3-q_2\\0\\q_3-q_2-q_1\\q_4+q_3+q_2\end{bmatrix}. $$
Since $q_4\gg q_3\gg q_2\gg q_1$, it is clear in this example that $\QQ_i-\QQ_{i+1}\in\mathbb{R}^+ m_i$ for all $i=1,2,3$. 
\end{example}

\section{Appendix: a bijection between $T$-paths and GCCs} \label{TPathsGCCs}
In this section, we first recall the construction of $T$-paths and the formula of cluster variables using $T$-paths as in \cite{S1}, then give a bijective proof of Theorem \ref{MainThm} via $T$-paths.

Let $P$ be a convex polygon with $\lqv+3$ vertices. 
Our initial triangulation of $P$ will consist of the set $T=\{T_1,\ldots,T_{\lqv}\} \cup \{T_{1,0},T_{1,1},T_{\lqv,0},T_{\lqv,1}\} \cup \{T_{i,i+1}:i\in[1,\lqv-1]\}$, where the first set is the set of diagonals and the last two sets are the set of boundary edges.

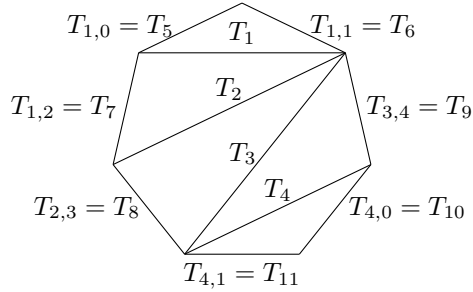
\begin{figure}[h!]
\begin{tikzpicture}[inner sep=3pt]
\tikzstyle{every node}=[font=\footnotesize]
\node[regular polygon, regular polygon sides=7, minimum size=3.5cm] at (0,0) (A) {};
\draw (A.corner 1) -- (A.corner 2) node[midway,left=3pt] {$T_{1,0}=T_5$};
\draw (A.corner 2) -- (A.corner 3) node[midway,left] {$T_{1,2}=T_7$};
\draw (A.corner 3) -- (A.corner 4) node[midway,left] {$T_{2,3}=T_8$};
\draw (A.corner 4) -- (A.corner 5) node[midway,below] {$T_{4,1}=T_{11}$};
\draw (A.corner 5) -- (A.corner 6) node[midway,right=2pt] {$T_{4,0}=T_{10}$};
\draw (A.corner 6) -- (A.corner 7) node[midway,right] {$T_{3,4}=T_9$};
\draw (A.corner 7) -- (A.corner 1) node[midway,right=3pt] {$T_{1,1}=T_6$};
\draw (A.corner 7) -- (A.corner 2) node[midway,above] {$T_1$};
\draw (A.corner 7) -- (A.corner 3) node[midway,above] {$T_2$};
\draw (A.corner 7) -- (A.corner 4) node[midway,left] {$T_3$};
\draw (A.corner 4) -- (A.corner 6) node[midway,above] {$T_4$};
\end{tikzpicture}
\caption{The initial triangulation of the quiver $(Q,Q')$ in Example \ref{QuiverW3Cycles}.}
\label{fig:triangulation example}
\end{figure}

The process of constructing the initial triangulation starts with choosing any vertex $\vf$ of $P$, labeling its two incident boundary edges as $T_{1,0}$ and $T_{1,1}$ and letting $T_1$ be the diagonal such that $T_{1,0}$, $T_{1,1}$ and $T_1$ form a triangle in the orientation as shown in Figure \ref{fig:initial triangulation} (the upper-left triangle).

Suppose that the diagonal $T_i$ ($i\in[1,\lqv-1]$) is drawn. The diagonal $T_{i+1}$ is obtained by rotating $T_i$ in the counterclockwise direction if $\delta_i=0$, in the clockwise direction if $\delta_i=1$. The boundary edge between $T_i$ and $T_{i+1}$ is labeled $T_{i,i+1}$.

When $i=\lqv$ then $T_{\lqv,1}$ is the boundary edge clockwise from $T_{\lqv}$ and $T_{\lqv,0}$ is the boundary edge counterclockwise from $T_{\lqv}$. Denote the common vertex of $T_{\lqv,0}$ and $T_{\lqv,1}$ by $\wf$.

\begin{figure}[h!]
\begin{tikzpicture}[inner sep=1pt,scale=1.3]
\coordinate (A) at (3,1);
\coordinate (B) at (0,1);
\coordinate (C) at (1.5,2);

\node [above] at (C) {$\vf$};
\draw (B) -- (C) node[midway,sloped,above] {$T_{1,0}$};
\draw (C) -- (A) node[midway,sloped,above] {$T_{1,1}$};
\draw (B) -- (A) node[midway,below=3pt] {$T_1$};

\begin{scope}[shift={(4,0)}]
\node [below] at (1.5,1) {$\wf$};
\draw (0,2) -- (3,2) node[midway,above] {$T_{\lqv}$};
\draw (0,2) -- (1.5,1) node[midway,sloped,above] {$T_{\lqv,1}$};
\draw (1.5,1) -- (3,2) node[midway,sloped,above] {$T_{\lqv,0}$};
\end{scope}

\begin{scope}[shift={(0,-2.5)}]
\coordinate (A) at (3,2);
\coordinate (B) at (1,1);
\coordinate (C) at (0,2);

\draw (C) -- (A) node[midway,above] {$T_{i}$};
\draw (C) -- (B) node[midway,left=2pt] {$T_{i,i+1}$};
\draw (B) -- (A) node[midway,below=6pt] {$T_{i+1}$};

\node at (1.5,0.25) {$\delta_i=0$};
\end{scope}

\begin{scope}[shift={(4,-2.5)}]
\coordinate (A) at (3,2);
\coordinate (B) at (0,2);
\coordinate (C) at (2,1);

\draw (B) -- (A) node[midway,above] {$T_{i}$};
\draw (B) -- (C) node[midway,below] {$T_{i+1}$};
\draw (C) -- (A) node[midway,right=6pt] {$T_{i,i+1}$};

\node at (1.5,0.25) {$\delta_i=1$};
\end{scope}
\end{tikzpicture}
\caption{Boundary edges and diagonals of an initial triangulation}
\label{fig:initial triangulation}
\end{figure}
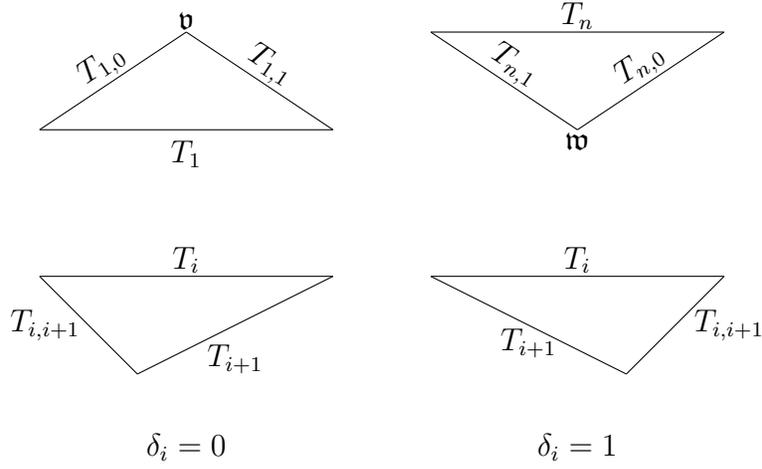

We can view both the snake diagram and the triangulation $T$ as graphs. Then there is a natural graph homomorphism $p$ between them satisfying
\begin{equation}\label{graph hom p}
p(T^{(i)}_j)=T_j,\quad p(T_{i,j})=T_{i,j}
\end{equation}
(for any $i,j$ that the equalities make sense). 
It is easy to check that the image of each vertex is uniquely determined using the requirement that a triangle in the snake diagram (with all the main diagonals added) must send to a triangle in $T$; so if the images of the three sides of a triangle are determined, the images of the three vertices are also determined (see Figures \ref{labels} and \ref{FirstLastTiles}).

In \cite{FZ2}, Fomin and Zelevinsky showed that the cluster variables of $\A(Q)$ are in bijection with the diagonals of the polygon $P$ where the set of initial cluster variables $\{x_1,\ldots,x_{\lqv}\}$ corresponds to $\{T_1,\ldots,T_{\lqv}\}$.

Let $M_{\vf,\wf}$ be the diagonal connecting $\vf$ and $\wf$, thus crossing the diagonals $T_1,\ldots,T_\lqv$. For $i\in[1,\lqv]$, let $p_i$ be the intersection of $M_{\vf,\wf}$ and $T_i$.

\begin{definition} \label{DefTPath} \cite{S1}
A $T$-path $\alpha$ from $\vf$ to $\wf$ is the sequence
$$\alpha=w_0\stackrel{T_{i_1}}{\longrightarrow}w_1\stackrel{T_{i_2}}{\longrightarrow}\cdots \stackrel{T_{i_{l(\alpha)}}}{\longrightarrow}w_{l(\alpha)}$$
such that

1) $\vf=w_0,w_1,\ldots,w_{l(\alpha)}=\wf$ are vertices of $P$.

2) $i_k\in\{0,1,\ldots,2\lqv+2\}$ such that $T_{i_k}$ connects the vertices $w_{k-1}$ and $w_k$ for each $k=1,2,\ldots,l(\alpha)$.

3) $i_j\ne i_k$ if $j\ne k$.

4) $l(\alpha)$ is odd.

5) $T_{i_k}$ crosses $M_{\vf,\wf}$ if $k$ is even.

6) If $j<k$ and both $T_{i_j}$ and $T_{i_k}$ cross $M$ then $p_{i_j}$ is closer to $\vf$ than $p_{i_k}$ is to $\vf$.
\end{definition}

Let $\P$ be the set of all $T$-path from $\vf$ to $\wf$. For any $\alpha\in\P$, let
\begin{equation}\label{xalpha}
x(\alpha)=\dis\prod_{k \text{ odd}}x_{i_k}\prod_{k \text{ even}}x_{i_k}^{-1}.
\end{equation}
Let ${\bf a}=(a_1,\dots,a_{n'})\in\{0,1\}^{n'}$ such that $a_i=1$ if and only if $i\in Q$. The following formula of the cluster variable $x[{\bf a}]$ is proved in \cite{S1}:
\begin{equation}\label{eq:T-path}
x[{\bf a}]=\sum_{\alpha\in\P} x(\alpha)
\end{equation}

\begin{definition}
We define a map $\psi_{\G,\P}:\G\to\P$ by sending  $\{S_{i,r}\}\in\G$ to the $T$-path $\alpha$ obtained by first constructing a path $\alpha'_1\alpha'_2\dots \alpha'_{2n+1}$ from $\vf$ to $\wf$ where $\alpha'_{2i}=T_{i}$ for $i\in[1,n]$,
$$\alpha'_{2i+1}=\left\{\begin{array}{lll}
T_{i}, & \text{if} & (|S_{i,1}|, |S_{i,2}|)=(\delta_i,1-\delta_i) \xh
T_{i+1}, & \text{if} & (|S_{i,1}|, |S_{i,2}|)=(1-\delta_i,\delta_i) \xh
T_{i,i+1}, & \text{if} & (|S_{i,1}|, |S_{i,2}|)=(0,0) \xh
\end{array}\right.$$
for $i\in[1,\lqv-1]$,  and

$\alpha'_1=\left\{\begin{array}{lll}
T_{1,0} & \text{if  } |S_{1,1+\delta_1}| = 1-\delta_1,\\
T_{1,1} & \text{otherwise},
\end{array}\right.$ \quad
$\alpha'_{2n+1}=\left\{\begin{array}{lll}
T_{\lqv,0} & \text{ if }  |S_{n-1,2-\delta_{n-1}}|=\delta_{n-1},\\
T_{\lqv,1} & \text{ otherwise},
\end{array}\right.$

\noindent then define $\alpha$ to be the path obtained from $\alpha'$ by canceling duplicate pairs.

We define $\psi_{\P,\G}:=\psi_{\G,\P}^{-1}: \P\to \G$. (As shown in the theorem below, $\psi_{\G,\P}$ is a bijection.) 
\end{definition}

\begin{theorem}\label{TpathsGCCs}
The maps $\psi_{\G,\P}$ is a well-defined bijection. Moreover, for $\alpha=\psi_{\G,\P}(\{S_{i,r}\})$,
$$\prod_{i=1}^n x_i^{-1}\prod_{i=0}^n y_i=x(\alpha),$$ thus $\psi_{\G,\P}$ induces a bijective proof of Theorem \ref{MainThm} using \eqref{eq:T-path}.
\end{theorem}
\begin{proof}
In order to prove Theorem \ref{TpathsGCCs}, we shall show that all maps below are bijective, and that  their composition is $\psi_{\G,\P}$:
$$\G \stackrel{\psi_{\G,\M}}{\longrightarrow} \M\stackrel{\L}{\longrightarrow} \{\textrm{complete $T$-paths from $\vf$ to $\wf$}\}\stackrel{\pi}{\longrightarrow} \P.$$

(i) We first define $\L$, which is exactly the folding map in \cite[\S4.3]{MS}. As defined in \cite{MS, S2}, a complete $T$-path $\alpha$ from $\vf$ to $\wf$ is similar to a $T$-path from $\vf$ to $\wf$ defined in Definition \ref{DefTPath}, in the sense that we require 1), 2), and
\begin{enumerate} \itemsep=5pt
\item[5')] the $2j$-th edge $T_{i_{2j}}=T_{j}$ (i.e., $i_{2j}=j$),
\item[6')] $T_{i_1}\le T_{i_2}\le \cdots$,
\end{enumerate}
where we use the order
\begin{equation}\label{edge order}
T_{1,0}<T_{1,1}<T_1<T_{1,2}<T_2<T_{2,3}<\cdots<T_n<T_{n,0}<T_{n,1}.
\end{equation}
Note that we do not require edges in $\alpha$ to be distinct. It is easy to see that a complete $T$-path has length $2n+1$.
For simplicity, we denote $\alpha$ using its edge sequence.
For $\gamma\in \M$, we define (recall that $p$ is defined in \eqref{graph hom p}):
\begin{equation}\label{Lgamma}
\L(\gamma)=L_1L_2\cdots L_{2\lqv+1},\quad  \textrm{  where
$L_{2j}=T_j$ for $j\in [1,n]$,\;  $L_{2j+1}=p(\gamma_j)$ for $j\in [0,n]$.}
\end{equation}
Note that the starting point of each $L_i$ is determined by $L_1\cdots L_{i-1}$.
The union of a perfect matching $\gamma$ with all the diagonals of tiles form a path $\alpha'_\gamma$ in the snake diagram.  If we consider the quotient map from the snake diagram to the triangulation of $P$, by identifying the diagonal edge $i$ with $T^{(i)}$ and identifying diagonal edge $i+1$ with $T^{(i)}_{i+1}$, then the image of $\alpha'_\gamma$ is the complete $T$-path $\L(\gamma)$.

(ii) We show that $\L$ has a well-defined inverse map $\L^{-1}$ (which is the unfolding map in \cite[\S4.5]{MS}), thus $\L$ is bijective. Indeed, $\L^{-1}$ sends a complete $T$-path $\theta=L_1\cdots L_{2n+1}$ to $\gamma=\{\gamma_1, \gamma_2, \ldots, \gamma_{\lqv}\}$, where
$\gamma_j$ is the unique edge in $Pl^{(j)}\cap p^{-1}(L_{2j+1})$, that is,
$\gamma_1=L_1$, $\gamma_\lqv=L_{2\lqv+1}$, and
$\gamma_j=T_{j}^{(j)}$ (resp. $T_{j+1}^{(j)}$, $T_{j,j+1}$) if  $L_{2j+1}$ is $T_j$ (resp. $T_{j+1}$, $T_{j,j+1}$) for $j\in[1,\lqv-1]$.

Next we show that $\gamma$ is indeed a perfect matching, it suffices to prove that the edges in $\gamma$ are disjoint, because it has the correct number ($=n+1$) of edges.
	
For $j, j'\in[0,n-2]$ with $j<j'$, $\gamma_j$ and $\gamma_{j'}$ are disjoint if $j'>j+1$ because $Pl^{(j)}$ and $PL^{(j')}$ are disjoint.  So we assume $j'=j+1$. We shall only discuss the case $\delta_j=\delta_{j+1}=0$ since other cases can be proved similarly. 	
\begin{figure}[h!]
\begin{tikzpicture}[MyPic]
\tikzstyle{every node}=[font=\small]
\draw [SoThick] (1,3) -- (4,1.5) node[midway,above=2pt] {$T_j$};
\draw [SoThick] (0,1.5) -- (4,1.5) node[midway,above] {$T_{j+1}$};
\draw [SoThick] (1,0) -- (4,1.5) node[midway,below=5pt] {$T_{j+2}$};
		
\draw (0,1.5) -- (1,3) node[midway,left] {$T_{j,j+1}$};
\draw (0,1.5) -- (1,0) node[midway,left] {$T_{j+1,j+2}$};

\begin{scope}[scale=1.6,shift={(4,0)}]
\tikzstyle{every node}=[font=\tiny]
\tsomepoints
		
\draw (P00) -- (P10);
\draw (P11) -- (P10) node[midway,above,sloped] {$T_{j,j+1}$};
\draw (P11) -- (P01) node[midway,above] {$T_{j+1}^{(j)}$};
\draw (P01) -- (P00);
\node at (0.5,0.5) [fill=lightgray] {$j$};
	
\draw (P10) -- (P20) node[midway,below] {$T_j^{(j)}$};
\draw (P20) -- (P21) node[midway,right] {$T_{j+2}^{(j+1)}$};
\draw (P21) -- (P11) node[midway,above] {$T_{j+1,j+2}$};
\node at (1.5,0.5) [fill=lightgray] {$j+1$};
		
\draw (P21) -- (P22);
\draw (P22) -- (P12);
\draw (P12) -- (P11) node[midway,left,near start] {$T_{j+1}^{(j+1)}$};
\node at (1.5,1.5) [fill=lightgray] {$j+2$};
\end{scope}
\end{tikzpicture}
\caption{Parts of the polygon and the snake diagram corresponding to the subquiver $j\to j+1\to j+2$}
\label{subpaths}
\end{figure}
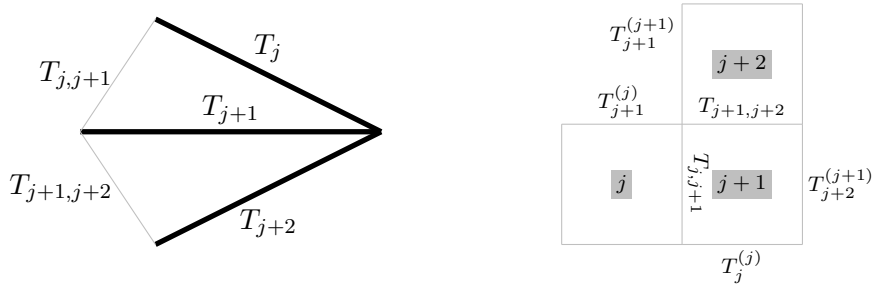
The subpath $L_{2j+1}L_{2j+2}L_{2j+3}L_{2j+4}$ of $\L(\alpha)$ is one of the following:
$$T_jT_{j+1}T_{j+1}T_{j+2}, \;
T_jT_{j+1}T_{j+1,j+2}T_{j+2},\;
T_{j,j+1}T_{j+1}T_{j+2}T_{j+2},\;
T_{j+1}T_{j+1}T_{j+2}T_{j+2}.$$
By looking at Figure \ref{subpaths}, we see that $\gamma_j$ and $\gamma_{j+1}$ are disjoint in each case.
	
(iii) We show that $\pi$ is bijective by giving its inverse $\pi^{-1}$. Suppose that $\alpha=T_{i_1}T_{i_2}\cdots T_{i_{l(\alpha)}}$ is a $T$-path from $\vf$ to $\wf$. If $n=1$, then $\alpha$ is already a complete $T$-path, so we define $\pi^{-1}(\alpha)=\alpha$. Now assume $n>1$. The sequence $\pi^{-1}(\alpha)=L=L_1L_2\cdots L_{2n+1}$ is obtained as a result of the following algorithm.
\vspace{-5pt}
\begin{enumerate} \itemsep=5pt
\item Initialize $L:=\alpha$.

\item Let $j$ run from $1$ to $n$: if $L_{2j}\neq T_j$, then insert $T_jT_j$ to $L$ so that the resulting $L$ is nondecreasing with the order given in \eqref{edge order}.

\item Define $\pi^{-1}(\alpha):=L$.
\end{enumerate}
We claim that $L$ is a complete $T$-path. Conditions 1) 2) 6') are obviously satisfied, and  5') can be proved by induction.

Combining (i)(ii)(iii) and Theorem \ref{PMsGCCs}, we have proved that $\psi_{\G,\P}$ is bijective.

Finally, we show that $\prod_{i=1}^n x_i^{-1}\prod_{i=0}^n y_i=x(\alpha)$. By the construction of $\pi^{-1}(\alpha)$ in (iii), $x(\alpha)$ remains unchanged if we replace $\alpha$ by the complete $T$-path $\pi^{-1}(\alpha)=T_{i_1}T_{i_2}\cdots T_{i_{2n+1}}$; this is because each time we insert the pair $T_jT_j$, the extra contribution to the product \eqref{xalpha} is $x_jx_j^{-1}=1$. So it suffices to show
$$\prod_{i=1}^n x_i^{-1}\prod_{i=0}^n y_i=\prod_{k \text{ even}}x_{i_k}^{-1}\prod_{k \text{ odd}}x_{i_k}.$$
By 5'), $\prod_{k \text{ even}}x_{i_k}^{-1}=\prod_{i=1}^n x_i^{-1}$, so it suffices to show
$\prod_{i=0}^n y_i=\prod_{k \text{ odd}}x_{i_k}$, or to show that $y_j=x_{i_{2j+1}}$ for $j\in [0,n]$.
Indeed, $T_{i_{2j+1}}=L_{2j+1}=p(\gamma_j)$ by \eqref{Lgamma}, thus $x_{i_{2j+1}}=w(\gamma_j)$ by the definition of the weight $w$ in \eqref{weight w}. Moreover, Theorem \ref{PMsGCCs} asserts that $w(\gamma_j)=y_j$. Thus $y_j=x_{i_{2j+1}}$.
\end{proof}

\begin{example} With the $T$-path $\alpha=T_5T_1T_9T_4T_{11}$, the complete $T$-path is $\pi^{-1}(\alpha)=T_5T_1T_2T_2T_3T_3T_9T_4T_{11}$. Then $\psi_{\P,\G}(\alpha)=((1,0),(1,0),(0,0))$ as you can see in Figure \ref{mapPhiPG}.
	
\begin{figure}[h!]
\begin{tikzpicture}[MyPic][scale=1.2]
\tikzstyle{every node}=[font=\scriptsize]

\node[regular polygon, regular polygon sides=7, minimum size=3cm] at (0,0) (A) {};
\draw (A.corner 1) -- (A.corner 2) node[midway,above,inner sep=3pt] {$T_5$};
\draw (A.corner 2) -- (A.corner 3) node[midway,left] {$T_7$};
\draw (A.corner 3) -- (A.corner 4) node[midway,left] {$T_8$};
\draw (A.corner 4) -- (A.corner 5) node[midway,below] {$T_{11}$};
\draw (A.corner 5) -- (A.corner 6) node[midway,right=2pt] {$T_{10}$};
\draw (A.corner 6) -- (A.corner 7) node[midway,right] {$T_9$};
\draw (A.corner 7) -- (A.corner 1) node[midway,above] {$T_6$};
\draw (A.corner 7) -- (A.corner 2) node[midway,above=0pt] {$T_1$};
\draw (A.corner 7) -- (A.corner 3) node[midway,above] {$T_2$};
\draw (A.corner 7) -- (A.corner 4) node[midway,left] {$T_3$};
\draw (A.corner 4) -- (A.corner 6) node[midway,above,inner sep=3pt] {$T_4$};

\draw [SoThick] (A.corner 1) -- (A.corner 2);
\draw [SoThick] (A.corner 2) -- (A.corner 7);
\draw [SoThick] (A.corner 7) -- (A.corner 6);
\draw [SoThick] (A.corner 6) -- (A.corner 4);
\draw [SoThick] (A.corner 4) -- (A.corner 5);

\begin{scope}[shift={(3,0)}]
\node (v1) at (0:0) {};
\node (v2) at (0:2) {};

\draw [|->,red] (v1) -- (v2) node[midway,above,font=\normalsize] {$\psi_{\P,\G}$};
\end{scope}

\begin{scope}[shift={(5.7,-0.5)}]
\tsomepoints
\tikzstyle{every node}=[font=\normalsize]

\draw [SoThick] (P00) -- (P10) node[midway,below] {$x_2$};
\draw (P10) -- (P11) node[midway,right] {$x_1$};

\draw [SoThick] (P20) -- (P30) node[midway,below] {$x_3$};
\draw (P30) -- (P31) node[midway,right] {$x_2$};

\draw (P40) -- (P41) node[midway,right] {$x_4$};
\draw (P40) -- (P50) node[midway,below] {$x_3$};	
\end{scope}
\end{tikzpicture}
\caption{An example of the map $\psi_{\P,\G}$}
\label{mapPhiPG}
\end{figure}
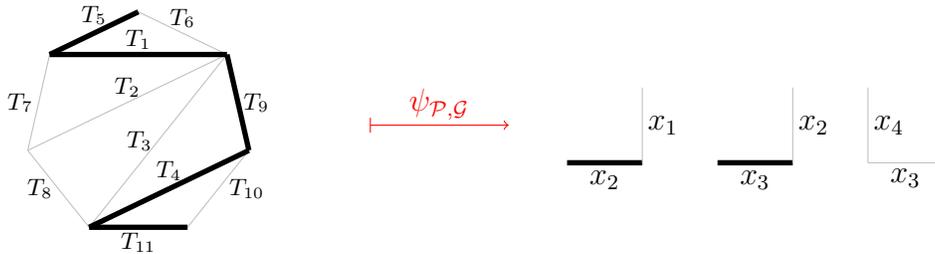
\end{example}

\end{document}